\documentclass{amsart}
\usepackage{amssymb}
\usepackage[polish,english]{babel}
\usepackage{lmodern}
\usepackage[utf8]{inputenc}
\usepackage[T1]{fontenc}
\usepackage{amsmath,amscd}
\usepackage{amsthm}
\usepackage[normalem]{ulem}

\usepackage{amsaddr}

\usepackage{xcolor}
\usepackage{graphicx}

\usepackage[section]{placeins}

\usepackage[hidelinks]{hyperref}
\usepackage{url}
\usepackage{orcidlink}

\newtheorem{theorem}{Theorem}
\newtheorem{proposition}{Proposition}
\theoremstyle{remark}
\newtheorem{remark}[theorem]{Remark}
\theoremstyle{definition}
\newtheorem{definition}[theorem]{Definition}

\usepackage{underscore}
\usepackage[ruled, vlined, linesnumbered, commentsnumbered, longend]{algorithm2e}

\title[Expansion in one-dimensional dynamics]{Rigorous computation of expansion in one-dimensional dynamics}

\thanks{This is the author's peer reviewed, accepted manuscript of the paper published in \textit{Chaos} 35, 123147 on December 31, 2025; \url{https://doi.org/10.1063/5.0287894}.}

\author{Pawe\l{} Pilarczyk$^{1,\ast}$ \orcidlink{0000-0003-0597-697X}}
\address{$^1$ Faculty of Applied Physics and Mathematics \& Digital Technologies Centre,
Gda\'{n}sk University of Technology,
ul.~Gabriela Narutowicza 11/12, 80-233 Gda\'{n}sk, Poland}

\address{$^\ast$ Corresponding author}

\author{Micha\l{} Palczewski$^{2}$ \orcidlink{0009-0007-3775-9589}}
\address{$^2$ Doctoral School \& Faculty of Applied Physics and Mathematics,
Gda\'{n}sk University of Technology,
ul.~Gabriela Narutowicza 11/12, 80-233 Gda\'{n}sk, Poland}

\author{Stefano Luzzatto$^{3}$ \orcidlink{0000-0001-6992-9190}}
\address{$^3$ Abdus Salam International Centre for Theoretical Physics (ICTP), Strada Costiera 11, 34151 Trieste, Italy}

\date{November 5, 2025}

\begin{document}


\begin{abstract}
We introduce an effective algorithmic method for the computation of a lower bound for uniform expansion in one-dimensional dynamics. The approach employs interval arithmetic and thus provides a rigorous numerical result (computer-assisted proof). The method uses efficient graph algorithms and an iterative approach for optimal performance. A software implementation of the method is made publicly available. This is an example of a quantitative result in the theory of dynamical systems, as opposed to many qualitative results whose assumptions may be difficult to verify and the conclusions may have limited use in practical models that describe natural phenomena.
We discuss and illustrate the effectiveness of our method and apply it to the quadratic map family.
\end{abstract}



\maketitle





\textbf{Most existing results in the theory of dynamical systems are of \emph{qualitative} rather than \emph{quantitative} nature: They start from qualitative assumptions and describe the dynamics in qualitative terms. Such assumptions may be very hard to verify in practice, and the conclusions may have limited use in the setting, for example, of dynamical systems which are used to describe natural phenomena. On the other hand, combining rigorous numerical methods with efficient algorithms and powerful computers opens new possibilities to obtain non-trivial \emph{quantitative} results that not only prove the existence of a qualitatively important feature of a given dynamical system, but also yield specific quantitative information about the phenomenon. In this paper, we focus on one such case; namely, we develop a carefully crafted method that combines rigorous numerics, an iterative approach, and classical graph algorithms to effectively and efficiently determine a lower bound on the rate of uniform expansion along trajectories in one-dimensional dynamics, a quantity especially important in determining and quantifying chaotic dynamics. We additionally prove that in many cases the bound we obtain is optimal.}

\section{Introduction}
\label{sec:intro}

The theory of dynamical systems has seen extraordinary development over the last century, and there is vast literature revealing the richness and subtlety of the theory. The field is motivated partly from physics, as dynamical systems can be seen as mathematical models of deterministic physical processes driven by physical laws, but also more intrinsically from natural mathematical questions about the iteration of maps in metric spaces or manifolds. Correspondingly, techniques for investigating dynamical systems cover a wide range of other areas, including algebra, geometry, functional analysis and probability, along many others.

In this paper, we employ rigorous numerical methods combined with efficient graph algorithms into an iterative method and software that provides a lower bound on the uniform expansion exponent in a one-dimensional dynamical system, as we state in Section~\ref{sec:unifexp}. As we discuss in Section~\ref{sec:motivation}, this method constitutes a very important ingredient in a computer-assisted proof of quantifying the amount of chaotic dynamics in a family of one-dimensional dynamical systems.

In this paper, we both provide algorithms as mathematical declarations and discuss specific implementation details, with emphasis on the limitations of real numbers that can be represented in the computer, including the use of correct rounding.



\subsection{Uniform expansion outside the critical neighbourhood}
\label{sec:unifexp}

Let \( I \) be an interval, \( \Delta \subset I \) be a subinterval (or a union of subintervals) of \( I \) and let \( f\colon I \to I \) be a map which is \( C^{1}\) on \( I \setminus \Delta\). 

\begin{definition}
\label{def:expC}
\( f \) is \emph{uniformly expanding outside} \( \Delta \) 
if there are constants \( C, \lambda >0 \) such that for any \( x\in I \) and \( n\geq 1 \) with \( f^{i}(x) \notin \Delta\) for all \( i=0,\ldots,n-1 \), we have 
\begin{equation}\label{eq:man}
|(f^{n})'(x)| \geq Ce^{\lambda n}.
\end{equation}
We call the highest such $\lambda$, if it exists, \emph{the expansion exponent of $f$ on $I \setminus \Delta$}.
\end{definition}
This seemingly quite technical property turns out to be important for a number of deep results in one-dimensional dynamics, especially when \( \Delta \) is a small neighbourhood of the critical points, in which case the derivative in \( I \setminus \Delta\) can be very small and the expansion condition \eqref{eq:man} is far from obvious. 
Nevertheless, the result of Ma\~n\'e \cite{Man85} says that for \( C^{2}\) maps condition \eqref{eq:man} holds quite generally, even for small critical neighbourhoods and even in the presence of an attracting periodic orbit (as long as \(\Delta \) contains the immediate basin of the points of the attractor which are closest to the critical points) and, in many cases, even for \emph{arbitrarily small critical neighbourhoods} \( \Delta\). 
Notice that \( C \) generally depends on \( \Delta \), since for \( n=1\) it needs to be chosen sufficiently small so that \( |f'(x)| \geq Ce^{\lambda}\) for all \( x\in I \setminus \Delta\), and in principle so does \( \lambda\), though it has been shown in \cite{PrzRivSmi03} that in many cases \( \lambda \) can be chosen independent of \( \Delta\).

Ma\~{n}\'{e}'s result \cite{Man85} is \emph{qualitative}, even though it is formulated in terms of the constants \( C\) and \( \lambda\), in the sense that it gives no information about their actual values, nor about the way in which they depend on \( \Delta\). Even \cite{PrzRivSmi03}, which gives conditions for \( \lambda \) to be independent of \( \Delta \), does not give any means of computing its value explicitly, nor indeed any way to actually verify that the required conditions are satisfied in particular examples. For many important applications, themselves also of a qualitative nature, this is enough as it is sufficient to know that the derivative grows exponentially outside \( \Delta\). 

However, there are many situations in which it would be useful to have explicit estimates on \( C, \lambda\) in order to obtain concrete and explicit numerical results about the dynamics. 
In Section \ref{sec:motivation} we give an in-depth discussion of a particular set of results concerning the existence of \emph{stochastic-like} dynamics in families of one-dimensional maps, and explain how condition \eqref{eq:man} plays a central role in those results and in particular is used to prove that in many families, stochastic-like maps occur for \emph{positive Lebesgue measure} sets of parameters. The lack of control over the values of the constants involved makes it, however, impossible to obtain any explicit estimates about the measure of this set of parameters. Moreover, most existing proofs apply to small neighbourhoods of especially good maps (for which in particular the constant \( \lambda \) is independent of \( \Delta\)), and in general it is impossible to verify that a given map satisfies such conditions. 


\subsection{Rigorous Explicit Estimates for Uniform Expansion}

Motivated by the comments above and the more detailed discussion in Section \ref{sec:motivation}, the purpose of this paper is to develop computational techniques which can \emph{rigorously prove} that a given map \( f\) is uniformly expanding outside some given critical neighbourhood \(\Delta\) and which can provide \emph{explicit rigorous} bounds for the values of the constants \(C, \lambda\).
We build our method on, and considerably improve, similar techniques proposed in the literature, in particular in \cite{Day2008,Golmakani2016}. We refine these techniques to obtain the \emph{ultimate} method for effective computation of \(\lambda\), and we show evidence of the \emph{optimality} of our approach. Moreover, we apply this method to a large-scale systematic study of the entire relevant parameter space in the quadratic map family.


\subsubsection{Background}
\label{sec:background}

The main idea, as developed in~\cite{Day2008}, is to split the set $I \setminus \Delta$ into a finite partition of intervals and to use a directed graph with weighted edges to represent the map $f$ along with estimates on its derivative on the partition elements. Using interval arithmetic to obtain rigorous bounds and classical graph algorithms, we can obtain rigorous estimates for \( \lambda\) and \( C \) in \eqref{eq:man}.

A major drawback of the method introduced in~\cite{Day2008} is its ``naive'' use of \emph{uniform partitions} which leads to high computational costs and sub-optimal results. The high computational cost is due to the fact that a very large partition must be created in order to obtain a tight estimate of $\lambda$, and the algorithm for the computation of this estimate is of time complexity $O(n^2)$ and requires $O(n^2)$ space, where $n$ is the number of partition elements. The latter problem was solved to a certain extent in~\cite{Golmakani2016} by applying the algorithm introduced in~\cite{Pilarczyk2020a} that uses $O(n)$ space, but still the time complexity makes obtaining tight estimates of expansion prohibitively expensive. It was shown in~\cite{Day2008} that partitions of several thousands of elements must be considered to obtain a valuable approximation of \( \lambda \) (see Figure~4 in~\cite{Day2008}).

Choosing a non-uniform partition may provide considerably better results, but the improvement resulting from using such partitions in \cite{Day2008} was inconsistent and hard to predict (as shown in Figure~5 in~\cite{Day2008}). This is due to the fact that choosing an optimal partition is a nontrivial task. A possible strategy, based on some heuristic arguments, was proposed in~\cite{Day2008} to make the partition finer where the derivative of $f$ was large. However, we will show in Section~\ref{sec:derivative} that this is not very effective.


\subsubsection{Our contribution}
\label{sec:introContrib}

In this paper, we propose \emph{two innovative and non-trivial methods} that, combined with the existing techniques, yield substantially better results. Indeed, in many cases we are able to prove that we get \emph{optimal} estimates, and we achieve such results with \emph{much less computational effort}.

The first innovation that we propose is an efficient and effective method, applicable in principle to a large class of one-dimensional maps, for \emph{constructing non-uniform partitions that consistently yield substantially better lower bounds on the expansion exponent than uniform partitions}.
The key point is that this construction is not based on a-priori considerations about where it would be more effective to have larger or smaller partition elements, but rather the construction of a \emph{dynamically-defined} partition. The key step in our approach is an algorithm that gradually improves an initially chosen coarse partition by splitting those partition elements that are responsible for computing a sub-optimal lower bound on the expansion exponent. In this way, our algorithm constructs a partition tailored to a specific map with the sizes of the subintervals determined in accordance to the actual dynamics.\footnote{The strategy of splitting intervals in order to define an optimal partition tailored to specific dynamics was also successful in \cite{Golmakani2020} where the parameter space was partitioned for the purpose of iterating certain intervals by the map for a long time. This confirms the general rule that dynamically adjusted parameters of numerical methods might provide superior results to any fixed or heuristically chosen parameters. Note that it is also the case with the choice of integration step in the well-known family of Runge-Kutta methods for the integration of differential equations.}
We show in Section~\ref{sec:new_vs_old} that, compared to uniform partitions with $10\,000$ elements, the non-uniform partitions consisting of as few as $1\,000$ elements consistently provide substantially better estimates of expansion for the quadratic map family at a fraction of the computational cost.
We also implement and apply our method to obtain explicit bounds for many parameter values, and for intervals of parameter values as well, in the one-dimensional quadratic family.

The second major technique that we introduce in this paper, see Section~\ref{sec:periorbit}, is a novel method to \emph{verify the extent to which our estimates are optimal}. This consists of a computational algorithm for proving the existence of a periodic orbit, lying fully outside \( \Delta \), for which the expansion admits a rigorous \emph{upper bound}. This proves that the constants cannot be improved by any further refinement of the partition and that, therefore, our method yields optimal results. This computation may also be incorporated into the dynamical construction of the partition mentioned in the previous paragraph as a way to determine an additional stopping criterion since it tells us that no further improvement is objectively possible.


\subsection{Formal Assumptions}
\label{sec:assumptions}

Our arguments hold and can be easily applied to quite general one-parameter families of one-dimensional maps. For simplicity we will focus here on the well known quadratic family, but for completeness we give here the precise formal minimal requirements. 

\subsubsection{Assumptions on the map}

Let $I \subset \mathbb{R}$ be a compact interval.
Let $\Delta \subset I$ be a finite union of open intervals with pairwise disjoint closures.
Let $\omega \subset \mathbb{R}$ be a compact interval or a set consisting of a single point.
For each $a \in \omega$, let $f_a \colon I \setminus \Delta \to I$ be a $C^1$ map with non-vanishing derivative. 
Notice that this is satisfied if $\Delta$ contains all the critical points of $f_a$, as well as other points on which $f_a$ or its derivative with respect to the space variable are not defined or not continuous. For example, if each $f_a$ is a $C^1$ map on $I$ and each $f_a$ has the same critical points then one could take $\Delta$ as an outer approximation of $\bigcup_{i=1}^k\left(c_i-\delta, c_i+\delta\right)$, where $\operatorname{Crit}(f_a)=\left\{c_1, \ldots, c_k\right\}$ and $\delta>0$ is an \textit{a priori} chosen number sufficiently small to make the closures of the intervals around the critical points disjoint. Note that we require that \(\Delta\) is independent of the parameter \(a \in \omega\).
This kind of a critical neighbourhood is implemented in our software discussed in Section~\ref{sec:software}.

\subsubsection{Assumptions on the numerical methods}

We assume that we have numerical methods, denoted by $F$, $DF$ and $F^{-1}$ that, given a compact interval $J \subset I \setminus \Delta$, provide rigorous outer bounds in terms of compact intervals for $f_a(x)$, $Df_a(x)$, and $f^{-1}_a(x)$, respectively, where $Df_a(x)$ denotes the derivative of $f_a$ with respect to $x$ at the point $x$. Formally, we assume the following:
\begin{eqnarray}
\label{eq:numF}
F (J) & \supset & \{f_a(x) : a \in \omega, \; x \in J\}, \\
\label{eq:numDF}
{DF} (J) & \supset & \{Df_a(x) : a \in \omega, \; x \in J\}, \\
\label{eq:numF-1}
{F^{-1}} (J) & \supset & \{x \in I \setminus \Delta : f_a(x) \in J \text{ for all } a \in \omega\}.
\end{eqnarray}

We remark that $I$ is a trivial candidate for $F(J)$ and $F^{-1}(J)$, and an interval containing $\{f_a'(x) : a \in \omega,\; x \in I\setminus\Delta\}$ is an easy to find candidate for $DF(J)$, but such bounds provide no meaningful information. However, in practice we are able to obtain very tight bounds on these quantities (using rigorous numerics discussed below) that provide us with very precise information.


\subsection{Rigorous numerics}
\label{sec:rigcomp}

Using the IEEE 754 standard double-precision floating-point numbers, combined with the concept of interval arithmetic, we can handle a wide family of one-dimensional maps satisfying the mild assumptions given above.
Due to the finite nature of computers, all calculations intended to be done on real numbers are conducted on a finite subset of representable numbers instead. In particular, the result of each operation must be rounded to a representable number, and even the calculation of elementary functions can only be conducted with limited precision. It is possible, however, to control the direction of rounding, and thus to obtain rigorous upper or lower bounds on the results of the operations.

Interval arithmetic (see e.g. \cite{Moore1966}) is a systematic approach that deals with rounding errors and other inaccuracies in computer calculations by representing numbers as intervals rather than single approximate values.
The fundamental idea is to perform arithmetic operations on intervals instead of individual numbers in such a way that the resulting interval encompasses all possible outcomes of the operation on the individual numbers contained in the given intervals. When conducting arithmetic operations such as addition or multiplication, this idea can be summarised as follows (with ``$\star$'' indicating a two-argument operation):
\[
[x_1,x_2] \star [y_1,y_2] \supset \{x\star y : x\in [x_1,x_2], y\in [y_1,y_2]\}.
\]
When evaluating a function, finite expansion into a series of fractions is typically applied, such as the Taylor series with an explicit bound for the error. This formula is processed using elementary arithmetic operations and yields an interval that encompasses the actual result.

Application of interval arithmetic to all calculations provides rigorous results in terms of outer approximations of the desired values. We highlight that the computational cost of interval arithmetic calculations is only slightly higher than the cost of regular (approximate) computations, because computing the result of an operation can usually be conducted by calculating the operation on the endpoints of the intervals with appropriate rounding direction.
Moreover, if we are interested in the lower or upper bound only, it is possible to decrease the cost even more by directly controlling the rounding direction of regular arithmetic operations.


\subsection{Overview of the paper}
\label{sec:overview}

In Section \ref{sec:graph} we give the definition and theory of \emph{weighted digraphs} and explain how these can be used to ``represent'' a map, including information about the derivative of the map, and to compute a lower bound \( \lambda\) on the expansion exponent based on a given partition of the domain of the map. 
This representation is not unique, and the result depends on the choice of partition; a different partition may result in finding a better lower bound for the same map.

In Section \ref{sec:digraph} we introduce our innovative algorithm designed to optimise the partition used to construct the weighted digraph and to compute a possibly tight lower bound on \( \lambda \).

In Section \ref{sec:analysis} we implement our method to compute lower bounds on the expansion exponent for many parameters in the quadratic family. 
We conduct extensive numerical analyses to compare our method against previous approaches, we discuss the limitations of our method and provide numerical evidence to justify that in some cases it is impossible to obtain substantially better results.
In particular, our computations show that partitions with around $1\,000$ elements constructed though our algorithms consistently provide substantially better estimates than uniform partitions with over $10\,000$ elements at a fraction of their computational cost.

In Section \ref{sec:periorbit} we introduce a method for proving the existence of a periodic orbit and computing an upper bound on its expansion exponent \( \lambda_{\max} \). This will be used to provide an upper bound for \( \lambda \) and to show the \emph{optimality} of our results. 

In Section \ref{sec:tecrem} we make some technical remarks about the algorithms and the computations. More specifically, in Section \ref{sec:complexity} we discuss the complexity of our algorithms in terms of their computational costs, and in Section \ref{sec:software} we give the details of the hardware and software used in the computations. 

In Section \ref{sec:motivation} we discuss the applications of the uniform expansion condition \eqref{eq:man} that motivate and justify the relevance and impact of the computational techniques developed in this paper. 

Finally, in Section \ref{sec:conclusions} we summarise the conclusions of our research and discuss some possible extensions and generalizations of our methods.


\section{Weighted Directed Graphs}
\label{sec:graph}

Effective algorithms on weighted directed graphs constitute the main tool in computing \( \lambda \) as in Definition~\ref{def:expC}.
In this section, we introduce the relevant definitions and algorithms, and we show how they can be used to compute \( \lambda \).


\subsection{Main definitions}
\label{sec:defgraph}

A \emph{weighted directed graph} (or a \emph{weighted digraph} for short) is a triple $G=(V, E, w)$, where $V$ is the finite set of \emph{vertices}, $E \subset V \times V$ is the set of \emph{edges}, and $w \colon E \to \mathbb{R}$ is the \emph{weight} function.
A \emph{path} in $G$ is a nonempty finite sequence of edges
\[
\gamma=\left(e_1, \ldots, e_n\right) \quad \text { such that } e_j=\left(v_j^0, v_j^1\right) \in E \text { and } v_j^1=v_{j+1}^0 .
\]
Vertices $v_1^0$ and $v_n^1$ are called the \emph{starting} and the \emph{ending} vertex of $\gamma$, respectively.
The number $n$ of edges in the sequence is called the \emph{length} of the path $\gamma$ and is denoted as $|\gamma|$.
We denote the set of all paths in $G$ by $\mathcal{P}(G)$.
The \emph{weight} and \emph{mean weight} of a path $\gamma=\left(e_1, \ldots, e_n\right) \in \mathcal{P}(G)$ are given respectively by 
\[
w(\gamma) := \sum_{j=1}^n w\left(e_j\right)
\quad \text{ and } \quad 
\overline{w}(\gamma) := \frac{w(\gamma)}{n}.
\]
The path $\gamma$ is called a \emph{cycle} if $v_1^0=v_n^1$.
We denote the set of all cycles in $G$ by $\mathcal{C}(G)$.
The minimum mean weight over all cycles in $G=(V, E, w)$ is
\[
\mu(G) :=\left\{\begin{array}{cc}
\min \{\overline{w}(\gamma): \gamma \in \mathcal{C}(G)\} & \text { if } \mathcal{C}(G) \neq \emptyset \\
+\infty & \text { if } \mathcal{C}(G)=\emptyset
\end{array}\right. .
\]
$G$ is a finite graph, which implies that the minimum is attained if $\mathcal{C}(G) \neq \emptyset$.

A weighted digraph $G$ is \emph{strongly connected} if for every pair of vertices $(v_\text{start},v_\text{end})$ such that $v_\text{start} \neq v_\text{end}$, there exists a path $\gamma$ in $G$ such that $v_\text{start}$ is its starting vertex, and $v_\text{end}$ is its ending vertex.


\subsection{Graph representation of a map}
\label{sec:graphRepr}

Let $f \colon I \to I$ be a self-map of a compact interval $I \subset \mathbb{R}$.
Let $\operatorname{Crit}(f)$ denote the set of all critical points of $f$ 
and all the points in which $f$ or its derivative are undefined or discontinuous if such points exist.
Let $\Delta$ be a finite union of open subintervals of $I$ with disjoint closures and containing $\operatorname{Crit}(f)$.
Assume that $f \colon I \setminus \Delta \rightarrow I$ is $C^1$. Note that under these assumptions, $Df$ is well defined on $I \setminus \Delta$ and is separated from zero.

\begin{definition}[$f$-admissible partition of $I \setminus \Delta$]
\label{def:partition}
A finite collection of closed intervals $\mathcal{I}=\left\{I_1, \ldots, I_k\right\}$ is an \emph{$f$-admissible partition} of $I \setminus \Delta$ if:
\begin{itemize}
    \item[(a)] $\operatorname{int}\left(I_i \cap I_j\right)=\emptyset$ for $i \neq j$,
    \item[(b)] $I_j \cap \operatorname{Crit}(f)=\emptyset$ for all $j$,
    \item[(c)] $I \setminus \Delta \subseteq \bigcup_{j=1}^k I_j$.
\end{itemize}
\end{definition}

\begin{definition}[Graph representation of a map]
\label{def:repr}
Given an $f$-admissible partition $\mathcal{I}$ of $I \setminus \Delta$, a weighted digraph $G=(V, E, w)$ is a \emph{representation} of $f$ on $I \setminus \Delta$ if:
\begin{itemize}
    \item [(a)] $V=\mathcal{I}$, 
    \item [(b)] $\left\{e=\left(I_1, I_2\right) \in V \times V : f\left(I_1\right) \cap I_2 \neq \emptyset\right\} \subset E$,
    \item[(c)] For each $e=\left(I_1, I_2\right) \in E$,
\[
w(e) \leq \min \left\{\ln |D f(x)|: x \in I_1 \cap f^{-1}\left(I_2\right)\right\}.
\]
\end{itemize}
\end{definition}

Note that, in particular, taking \( E \) to be the graph with edges connecting each vertex to all other vertices, trivially
satisfies Definition~\ref{def:repr} but provides very little information about $f$. We aim at choosing as few edges as possible. Similarly, we make an effort to provide the estimates $w(e)$ as high as possible.

Estimating the minimum accumulated derivatives can be reduced to calculating the weights of certain paths in a graph $G$ that represents the map $f$ on $I \setminus \Delta$.
Indeed, given a point $x \in I \setminus \Delta$ and a path $\gamma=\left(e_1, \ldots, e_n\right) \in \mathcal{P}(G)$ such that $e_i=\left(I_{i-1}, I_i\right)$ for all $i=1, \ldots, n$, and $f^j(x) \in I_j$ for all $j=0, \ldots, n$, we have
\begin{equation}\label{eq:expPath}
\ln \left|D f^n(x)\right|=\sum_{j=0}^{n-1} \ln \left|D f\left(f^j(x)\right)\right| \geq w(\gamma)
\end{equation}
and therefore
\begin{equation}
\label{eq:expMean}
\tilde{\lambda}_n(x) := \ln \left|Df^n(x)\right|^{\frac1n} \geq \overline{w}(\gamma).
\end{equation}


\begin{proposition}[{see \cite[Proposition 1]{Day2008}}]
\label{prop:expCycle}
Let $G$ be a representation of a map $f \colon I \to I$ that has at least one cycle. Let $\lambda = \mu(G)$. Then $\lambda$ is a lower bound on the expansion exponent of $f$ on $I \setminus \Delta$.
\end{proposition}

\begin{proof}
Consider the set $\mathcal{S}(G)$ of all the \emph{simple paths} in $G$, that is, paths that do not contain cycles (or, in other words, paths with pairwise different vertices). This is a finite set because the number of vertices of $G$ is finite. Note that every path in $G$ can be decomposed into a simple path and a finite number of cycles. For each trajectory segment that follows a cycle in $G$, it follows from \eqref{eq:expMean} that \eqref{eq:man} is satisfied with $C = 1$. However, we need to find $C$ small enough so that \eqref{eq:man} will also be satisfied for other paths. Let us thus take a number $C \leq 1$ that satisfies
\begin{equation}
\label{eq:Cexp}
C\leq \exp\left(\min\{ w(\Gamma)-|\Gamma|\lambda:\Gamma\in\mathcal{S}(G)\}\right).
\end{equation}
Combining \eqref{eq:Cexp} with \eqref{eq:expMean} and substituting it into \eqref{eq:man} yields Proposition~\ref{prop:expCycle}.
\end{proof}


\subsection{Minimum cycle mean}
\label{sec:Karp}

Consider a weighted digraph $G=(V, E, w)$ with $n$ vertices and $m$ edges. The algorithm introduced by Karp~\cite{Karp1978} computes the minimum mean weight of all possible cycles (also called the \emph{minimum cycle mean}) in $G$ in $O(nm)$ time and $O(n^2)$ space. We briefly recall this algorithm here and introduce an additional step based on \cite{Chaturvedi2017} that makes it possible to retrieve a cycle that realises this minimum. Note that Karp's original suggestion in \cite{Karp1978} for extracting this cycle was wrong; therefore, we use the approach introduced in~\cite{Chaturvedi2017}.

Assume $G$ is strongly connected. Otherwise, one of the classical algorithms can be used to split $G$ into strongly connected components in linear time.
Note that every cycle is contained in one strongly connected component, so taking the minimum of the minimum cycle means computed for the strongly connected components would yield the global minimum valid for the entire graph $G$.

Select an arbitrary vertex $s$. For every vertex $v \in V$ and each non-negative integer $k$, define $F_k(v)$ as the minimum weight of a path of length $k$ with its starting vertex $s$ and its ending vertex $v$. If no such path exists then set $F_k(v) := \infty$. Theorem~1 in~\cite{Karp1978} says that
\begin{equation}
\label{eq:lambdaKarp}
\lambda := \min \left\{\overline{w}(C) : \text{$C$ is a cycle in $G$}\right\} = \min _{v \in V} \max _{0 \leq k \leq n-1}\frac{F_n(v)-F_k(v)}{n-k},
\end{equation}
where the minimum is taken over those vertices $v \in V$ for which $F_n(v) < \infty$. Recall that $n = \operatorname{card} V$.

Our adaptation of Karp's algorithm comprises of three stages. First, an arbitrary vertex $s \in V$ is chosen, and the quantities $F_k(v)$ are computed in subsequent iterations, starting from $k=0$. While determining a minimum-weight path from $s$ of length $k$ to any vertex $v \in V$, the predecessor vertex $u = P_k(v)$ is additionally recorded based on the last edge $e=(u,v)$ in the path. This extra step is necessary to determine a cycle at which the minimum mean weight is attained in the third stage. In the second stage, the quantity in \eqref{eq:lambdaKarp} is computed.
Additionally, a vertex $v_0 \in V$ that realises the minimum in~\eqref{eq:lambdaKarp} is determined. In the third stage, a cycle $\gamma$ is determined that realises $\lambda$ by means of tracing the path stored in $P$ backwards, starting from $v_0$. The details are shown in Algorithm~\ref{alg:Karp} below.

\begin{algorithm}
    \label{alg:Karp}
    \caption{Minimum cycle mean}
    \SetKw{KwDownTo}{down to}
    \SetKwInOut{Input}{input}
    \SetKwInOut{Output}{output}
    \Input{
        $G = (V,E,w)$: a strongly connected weighted digraph 
    }
    \Output{
        $\lambda$ satisfying \eqref{eq:lambdaKarp} for the graph $G$; \\
        $\gamma$: a cycle in $G$ with the mean weight $\lambda$
    }

    \emph{--- First Stage ---}

    take any $s\in V$ (a starting vertex)

    \ForEach{$v \in V \setminus \{s\}$}{$F_0(v) := \infty$}

    $F_0(s) := 0$
    
    $n := |V|$

    \For{$k := 1$ \KwTo $n$}{
        \ForEach {$v \in V$}{$F_k(v) := \infty$}
        \ForEach{$e = (u,v) \in E$}{
            \If{$F_{k-1}(u) + w(e) < F_{k}(v)$}{
                $F_{k}(v) := F_{k-1}(u) + w(e)$\;
                $P_k(v) := u$          \label{line:defPk}
            }
        }
    }

    \emph{--- Second Stage ---}

    $v_n := \emptyset$

    \ForEach{$v \in V$}{
        \If{$F_n(v) < \infty$}{
            $M(v) := \max_{0 \leq k \leq n-1}(F_n(v)-F_k(v))/(n-k)$         \label{line:defMv}

            \If{$v_n = \emptyset$ or $M(v) < M(v_n)$}{
                $v_n := v$        \label{line:defvn}
            }
        }
    }

    $\lambda := M(v_n)$

    \emph{--- Third Stage ---}

    \For{$k := n$ \KwDownTo $1$}{
        $v_{k-1} := P_k (v_k)$
    }

    $\Gamma := \big((v_0,v_1), \ldots, (v_{n-1},v_{n})\big)$

    let $\gamma$ be any cycle contained in the path $\Gamma$        \label{line:defgamma}

    \KwRet{$\lambda$, $\gamma$}
\end{algorithm}

\begin{proposition}
\label{prop:Karp}
Let $G$ be a strongly connected weighted digraph
with at least one edge.
Then the number $\lambda$ computed by Algorithm~\ref{alg:Karp} is the minimum cycle mean in $G$ and the path $\gamma$ is a cycle that realises this minimum.
\end{proposition}

\begin{proof}
Since the numbers $F_k(v)$ computed in Algorithm~\ref{alg:Karp} align with $F_k(v)$ in Karp's original algorithm, it follows directly from \cite[Theorem~1]{Karp1978} that the algorithm effectively computes the minimum cycle mean $\lambda$ of $G$.

The vertex $v_n$ determined in the second stage of the algorithm (line~\ref{line:defvn}) is the vertex minimising the expression in \eqref{eq:lambdaKarp}, because this is how $M(v)$ is defined (line~\ref{line:defMv}).
Observe that each $P_k(v)$ defined in the algorithm (line~\ref{line:defPk}) keeps the predecessor of the vertex $v$ on a minimising path of length $k$ to~$v$.
Therefore, the path $\Gamma$ constructed in the third stage is an edge progression of length $n$ to~$v$ whose weight is equal to $F_n(v)$.
Since the length of $\Gamma$ is $n$, which means that it consists of $n+1$ vertices in a graph that has only $n$ vertices, at least one vertex must appear twice in this path. Therefore, the path $\Gamma$ contains at least one cycle, and thus the cycle $\gamma$ is defined correctly (line \ref{line:defgamma}).
It follows from \cite[Lemma~1]{Chaturvedi2017} that $\gamma$ is a cycle of minimum mean weight.
\end{proof}

\begin{remark}
In a practical numerical implementation, when setting appropriate rounding directions in the calculations conducted in Algorithm~\ref{alg:Karp}, it is possible to compute a rigorous lower bound on $\lambda$, which we have actually implemented in our software. However, due to the inaccuracies that come from rounding errors, the actual cycle $\gamma$ found might not be a minimising one. Nevertheless, the mean weight of $\gamma$ is very close to $\lambda$, and we shall only need this cycle to split the corresponding intervals, not as part of rigorous results. Therefore, this minor discrepancy does not cause an obstacle to our method.
\end{remark}


\section{Main algorithms}
\label{sec:digraph}

In order to compute a lower bound \( \lambda \) for the expansion exponent of a one-dimensional map, we construct an admissible partition of \( I \setminus \Delta \), a graph representation of the map, and we use algorithms introduced in the previous sections. We suppose throughout this section that we are in the setting described in Section~\ref{sec:assumptions}.


\subsection{Algorithmic construction of a graph representation of a map}
\label{sec:digraphConstr}

Let \(\mathcal{I}=\left\{I_1, \ldots, I_k\right\}\) be a partition of \(I \setminus \Delta\) that is \(f_a\)-admissible for all $a \in \omega$. Define the set of vertices of the graph as
\begin{equation}
\label{eq:constrV}
V := \mathcal{I}
\end{equation}
and use the numerical method $F$ satisfying \eqref{eq:numF} to define the set of edges as
\begin{equation}
\label{eq:constrE}
E:=\left\{\left(I_i, I_j\right) \in V \times V : F\left(I_i\right) \cap I_j \neq \emptyset\right\}.
\end{equation}
For each edge \(e=\left(I_i, I_j\right) \in E\), a closed interval \(L(e) \subset I_i\) should be determined so that
\[
\left\{x \in I_i : f_a(x) \cap I_j \neq \emptyset \text{ for some } a \in \omega \right\} \subset L(e).
\]
Note that \(L(e)=I_i\) is a valid choice, but a smaller interval can be obtained as $I_i \cap F^{-1}(I_j)$, using the method $F^{-1}$ satisfying~\eqref{eq:numF-1}.
$DF(L(e))$ can be computed using the method $DF$ satisfying~\eqref{eq:numDF}, and the lower endpoint of the interval $\ln |DF(L(e))|$ can be computed using interval arithmetic in order to obtain a representable number \(b(L(e))\) such that
\[
b(L(e)) \leq \min \left\{\ln \left|D_x f(x, a)\right|: x \in L(e), a \in \omega\right\}.
\]
One can then define
\begin{equation}
\label{eq:constrw}
w(e):=b(L(e)).
\end{equation}
It is easy to see that the graph $G = (V, E, w)$ defined by \eqref{eq:constrV}, \eqref{eq:constrE} and \eqref{eq:constrw} satisfies Definition~\ref{def:repr} of a representation of the map $f_a$ for every $a \in \omega$.

\begin{remark}
In the original approach \cite{Day2008}, the set $\Delta$ was added as an additional vertex to the graph $G$. However, edges leaving $\Delta$ were not computed, so this vertex would never appear in any cycle. Therefore, adding this vertex or not does not affect the estimates discussed in the current paper.
\end{remark}


\subsection{Iterative construction of an admissible partition}
\label{sec:partConstr}

Uniform admissible partitions (consisting of intervals of approximately the same size) were mainly considered in \cite{Day2008,Golmakani2020,Golmakani2016,Pilarczyk2020a}. We provide Algorithm~\ref{alg:refine} for creating a non-uniform partition that yields considerably better expansion estimates.

\begin{algorithm}[htbp]
    \label{alg:refine}
    \caption{Selective partition refinement}
    \SetKwInOut{KwIn}{Input}
    \SetKwInOut{KwOut}{Output}

    \KwIn{
        $\mathcal{I}^0 = \{I^0_1, \ldots, I^0_{k_0}\}$: an $f$-admissible partition of $I \setminus \Delta$; \\
        $K \geq k_0$: the maximum partition size allowed; \\
        $\xi > 0$: a threshold of interval width to trigger Algorithm~\ref{alg:periorbit} \\ ($\xi = 0.01$ by default); \\
        $n_0 > 0$, $\eta \geq 1$: the number of steps in which the improvement \\ by a factor of $\eta$ is desired ($n_0=10$ and $\eta=10^{-10}$ by default)
    }
    \KwOut{
        $\mathcal{I} = \{I_1, \ldots, I_{k}\}$: an $f$-admissible partition of $I \setminus \Delta$ with $k \leq K$
    }

    $\mathcal{I} := \mathcal{I}^0$

    \While{$\operatorname{card} \mathcal{I} \leq K$}{
        construct a graph representation $G=(V,E,w)$ of $f$ on $I \setminus \Delta$ with $V=\mathcal{I}$ as~described in Section~\ref{sec:digraphConstr}\;

        apply Algorithm~\ref{alg:Karp} to all the strongly connected components of $G$ in order to find the minimum cycle mean $\lambda$ in $G$ and a cycle $\gamma$ that realises this minimum\;

        \If{the width of all the intervals in the cycle $\gamma$ is below $\xi |I|$}{
            apply Algorithm~\ref{alg:periorbit} to obtain $\lambda_{\max}$; \label{line:lambdaMax}
        }

        \ForEach{vertex $v$ in the cycle $\gamma$}{
            denote the interval in $\mathcal{I}$ corresponding to $v$ by $[p_l,p_r]$\;
            let $p$ be a representable number closest to $(p_l + p_r)/2$\;
            \If{$p_l < p < p_r$}{
                replace $I_j$ in $\mathcal{I}$ with two intervals: $[p_l,p]$ and $[p,p_r]$\;
                \If{$\operatorname{card} \mathcal{I} = K$}{
                    \KwRet{$\mathcal{I}$}
                }
            }
        }

        \If{no interval was split in the above ``foreach'' loop \textbf{or} the ``while'' loop was already run more than $n_0$ times and $\lambda \leq \eta \lambda'$, where $\lambda'$ is the value of $\lambda$ computed before $n_0$ runs of the ``while'' loop}{
            \KwRet{$\mathcal{I}$}
        }
    }

    \KwRet{$\mathcal{I}$}
\end{algorithm}

The key idea of our method is to begin with a coarse partition of $I \setminus \Delta$ and then iteratively: (1) apply Algorithm~\ref{alg:Karp} to obtain a cycle $\gamma$ with minimum mean weight, (2) split the partition intervals on $\gamma$ in half in the hope that this cycle was an artefact due to overestimates and it will no longer appear after the increase in accuracy of the representation of the map. By eliminating cycles with low mean weight, this strategy is expected to consistently improve the bounds for expansion obtained in consecutive steps. We stop the iterations if the desired size of the partition is achieved or no substantial increase in the computed lower expansion bound $\lambda$ is obtained in several steps. We illustrate this process with a specific example in Section~\ref{sec:iterSubdiv}.


\subsection{Computation of a lower bound for the expansion}
\label{sec:compExp}

Finally, Algorithm~\ref{alg:main} describes the entire procedure of obtaining a lower bound on the expansion exponent of $f$ on $I \setminus \Delta$. As we show in Section~\ref{sec:new_vs_old}, this method provides considerably better results than the method based on uniform partitions in~\cite{Day2008}.

\begin{algorithm}[htbp]
    \label{alg:main}
    \caption{Expansion exponent of a map}
    \SetKwInOut{Input}{input}
    \SetKwInOut{Output}{output}
    \Input{
        $I \subset \mathbb{R}$, $\Delta \subset I$, $\omega \subset \mathbb{R}$: sets satisfying assumptions set forth in Section~\ref{sec:assumptions} for a given family of maps $f_a$; \newline
        $F$, $DF$, $F^{-1}$: numerical methods satisfying \eqref{eq:numF}, \eqref{eq:numDF}, and \eqref{eq:numF-1}, respectively\; \\
        $K > 0$: the maximum partition size allowed
    }
    \Output{
        $\lambda$: a lower bound on the expansion exponent of $f_a$ on $I \setminus \Delta$ as in Definition~\ref{def:expC}, valid for all $a \in \omega$;
    }

    $k_0 := \max\{9, \text{number of intervals that form $I \setminus \Delta$}\}$\;    \label{line:take_k_0}

    let \(\mathcal{I}^0=\left\{I^0_1, \ldots, I^0_{k_0}\right\}\) be an \(f\)-admissible partition of \(I \setminus \Delta\) consisting of intervals of approximately the same size;

    let \(\mathcal{I}=\left\{I_1, \ldots, I_k\right\}\) be an \(f\)-admissible partition of \(I \setminus \Delta\) returned by Algorithm~\ref{alg:refine} applied to $\mathcal{I}^0$ and $K$;   \label{line:partition}

    let $G=(V,E,w)$ be a graph representation of $f$ on $I \setminus \Delta$ with $V=\mathcal{I}$ constructed following the description in Section~\ref{sec:digraphConstr}\;

    let $\lambda$ be the number returned by Algorithm~\ref{alg:Karp} applied to the graph $G$;

    \KwRet{$\lambda$}
\end{algorithm}

\begin{proposition}
\label{prop:main}
Let $I \subset \mathbb{R}$ be a compact interval, let $\Delta \subset I$ be a finite union of open intervals with pairwise disjoint closures, and let $\omega \subset \mathbb{R}$ be a compact interval.
For each $a \in \omega$, let $f_a \colon I \setminus \Delta \to I$ be a $C^1$ map with non-vanishing derivative. Assume that $F$, $DF$ and $F^{-1}$ are numerical methods that provide rigorous outer bounds in terms of compact intervals for $f_a(x)$, $Df_a(x)$, and $f^{-1}_a(x)$, respectively, as defined by \eqref{eq:numF}, \eqref{eq:numDF} and \eqref{eq:numF-1}. Finally, let $K \in \mathbb{N}$.
Then the quantity computed by Algorithm~\ref{alg:main} applied to $I, \Delta, \omega, F, DF, F^{-1}$ and $K$ 
is either $+\infty$
or a number $\lambda \in \mathbb{R}$
that is a lower bound on the expansion exponent of $f$ on $I \setminus \Delta$, as in~\eqref{eq:man}.
Moreover, each number $\lambda_{\max}$, if any, obtained during the computations, is an upper bound on the expansion exponent of $f$ on $I \setminus \Delta$, as in \eqref{eq:man}.
\end{proposition}

\begin{proof}
Let $k_0$, $\mathcal{I}^0$, $\mathcal{I}$, $G$, $\lambda_l$ and $\lambda_c$ be as defined in Algorithm~\ref{alg:main}.
It follows from Proposition~\ref{prop:Karp} that $\lambda$ returned by Algorithm~\ref{alg:Karp} is the minimum cycle mean in $G$ or $+\infty$ if $G$ has no cycles. As a consequence, if Algorithm~\ref{alg:main} returns a number $\lambda < +\infty$ then the conclusion follows from Proposition~\ref{prop:expCycle}.
The ``Moreover'' part follows immediately from Proposition~\ref{prop:periorbit}.
\end{proof}

\begin{remark}
The case of Algorithm~\ref{alg:main} returning $+\infty$ should be considered a failure of the method in providing meaningful information
about long-term expansion in the analysed dynamical system.
This may happen if the constructed graph representation $G$ of the dynamics contains no cycles, which implies that every trajectory leaves $I \setminus \Delta$
after a finite number of iterations.
\end{remark}


\section{Implementation and computations}
\label{sec:analysis}

In this section, we introduce implementation of the rigorous computational techniques and algorithms described above, and we describe and discuss results of some specific computations.
We will focus on the quadratic family 
\begin{equation}\label{eq:quad}
f_{a}(x) = a - x^{2} 
\end{equation}
for parameters \( a \) belonging to the parameter interval 
\begin{equation}\label{eq:Omega}
\Omega = [1.4, 2]
\end{equation}
since for all \( a\notin \Omega\) the dynamics is essentially well understood.\footnote{We could define \( \Omega = [\tilde a, 2] \) where \( \tilde a\) is the \emph{Feigenbaum parameter} whose value is known only approximately as \( a^{*}\approx 1.4011...\) but the difference is not so relevant for our discussion. } We emphasise that our main purpose is the development of the methods and algorithms described in Sections \ref{sec:graph} and \ref{sec:digraph} above and Section \ref{sec:periorbit} below, that can be applied to much more general one-dimensional maps satisfying the mild conditions given in Section~\ref{sec:assumptions}. The computations we describe below are intended as an illustration of the effectiveness of our method, although they also provide important data for future application to explicit parameter-exclusion arguments as discussed in Section~\ref{sec:motivation}. 

For formal completeness, let us note that the quadratic family \eqref{eq:quad} meets the assumptions listed in Section~\ref{sec:assumptions}. Indeed, $f_a$ is of class $C^1$, and its derivative $f_a'(x) = -2x$ vanishes only on $x=0$ independent of $a$. Take $\Delta := (-\delta,\delta)$ for a small value of $\delta>0$. Since $f_a$ is invertible on $\{x\leq 0\}$ as well as on $\{x\geq 0\}$, with its inverse defined by $x=\pm\sqrt{a-y}$, one can use the explicit formulas combined with interval arithmetic to obtain numerical methods $F$, $DF$ and $F^{-1}$ satisfying \eqref{eq:numF}, \eqref{eq:numDF} and \eqref{eq:numF-1}, respectively.

The code of the software implementation of our algorithms is provided in \cite{www} together with scripts and data. Raw results of the computations described in this section are also available in \cite{datasets}.


\subsection{Comparison of the new method against the previous approach}
\label{sec:new_vs_old}

In order to show the improvement of our method in comparison to the previous approach proposed in~\cite{Day2008}, let us first consider the quadratic map \eqref{eq:quad} with $a \in [1.99999,2]$ that was analysed in \cite{Golmakani2016} outside the critical neighbourhood $\Delta = (-0.001,0.001)$. By taking as many as $100\,000$ partition elements, the authors of \cite{Golmakani2016} obtained a lower bound for \( \lambda\)  slightly above $0.49$ at a high computational cost, which corresponds to 28 minutes of running on one thread of the Intel\textregistered{} Core\texttrademark{} i7-13700H processor. Our new approach proves that $\lambda \geq 0.58$ in a few seconds' time on the same processor, and we additionally prove that the best lower bound on $\lambda$ one could hope for is about $0.62$ due to the existence of an unstable periodic orbit for $a=1.99999$ with the upper bound on its expansion exponent of slightly less than $0.62705$. In fact, when considering $a=1.99999$ in the computation instead of the interval $[1.99999,2]$, our new method provides a lower bound for $\lambda$ of a little more than $0.62702$, which is nearly optimal, while the previous method still yields $0.49$.

For a more comprehensive comparison, we took $1025$ equally-spaced values $a_i \in \Omega$, including the endpoints: $a_i := 1.4 + 0.6 i / 1024$ for $i = 0, \ldots, 1024$. The reason for taking specifically $1024$ as the maximum index $i$ (and not taking, for example, $1000$) was that we chose to use binary subdivision of the entire parameter interval $\Omega = [1.4,2]$ in the software, and here we subdivided $\Omega$ into $2^{10} = 1024$ segments. This is in preparation for adaptive partitioning of the parameter interval with varying subdivision depth, as discussed further in Section~\ref{sec:adaptiveParam}.

For each $a_i$, we computed two quantities.
First, we applied Algorithm~\ref{alg:main} to $I = [-2,2]$, $\Delta = (-\delta,\delta)$ with $\delta = 0.001$, $\omega = \omega_i = \{a_i\}$, the numerical methods $F$, $DF$ and $F^{-1}$ satisfying \eqref{eq:numF}, \eqref{eq:numDF}, and \eqref{eq:numF-1}, respectively, and $K = 1\,000$.
Second, we performed analogous computations with a modification of Algorithm~\ref{alg:main} in which we replaced the construction of a partition with selective refinement (lines \ref{line:take_k_0}--\ref{line:partition}) by taking a uniform partition consisting of $k=10\,000$ (that is, $k=10^{4}$) intervals.

Then we compared the two quantities; see Figure~\ref{fig:new_vs_old}. The new method clearly provides higher lower bounds for the expansion exponent.
%
\begin{figure}[htbp]
\centering
\includegraphics[width=\textwidth]{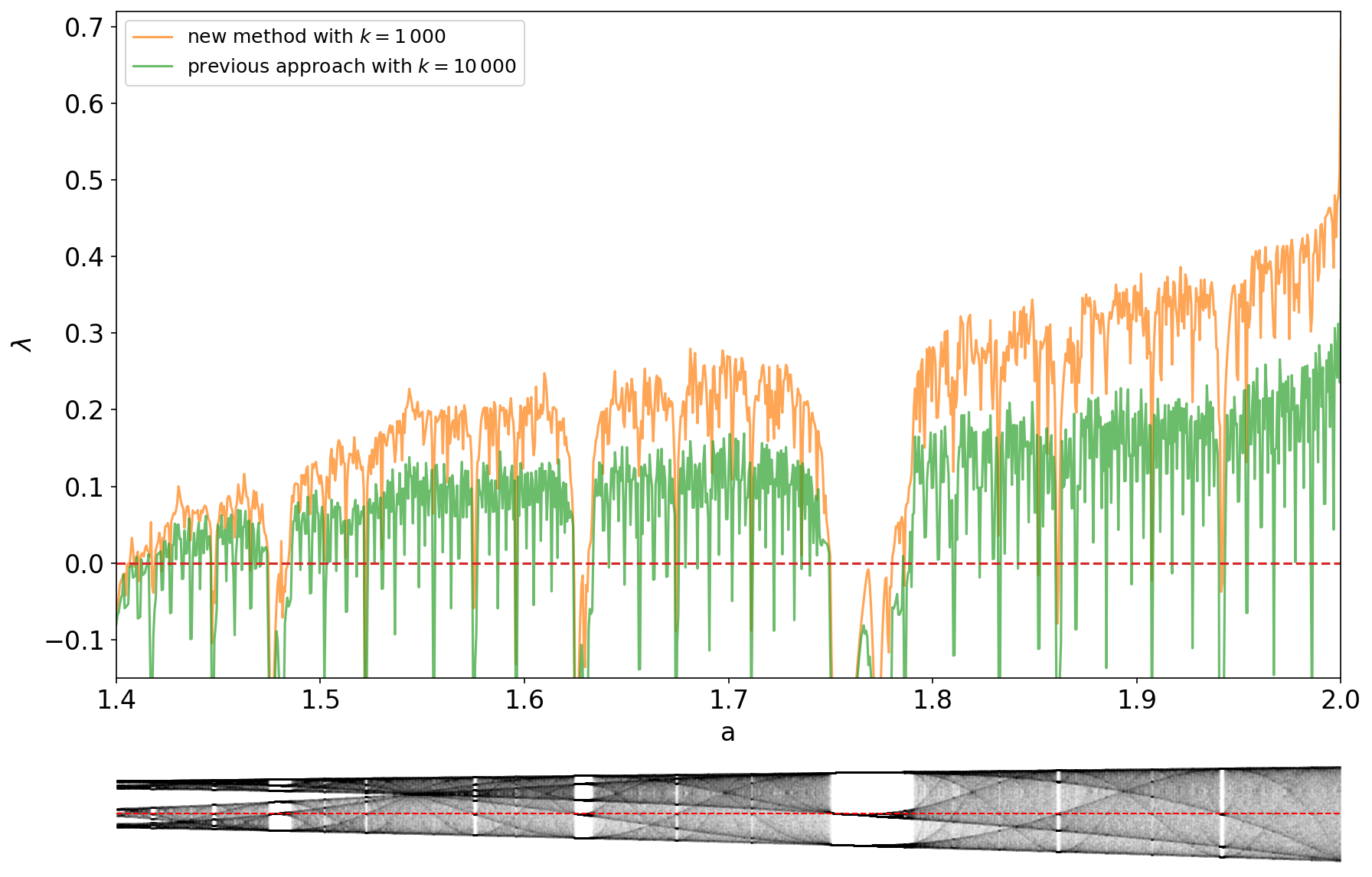}
\caption{\label{fig:new_vs_old}
The lower bound $\lambda$ on the expansion exponent computed for the quadratic map $f_a$ with the new method and with the previous approach, as discussed in Section~\ref{sec:new_vs_old}.
The bifurcation diagram is shown for reference. (Colour online.)}
\end{figure}
We checked that in no case did the earlier method yield an estimate for $\lambda$ higher than the new approach.
Specifically, the differences between the estimates obtained using our method and the previous approach were in the range $[0.00055,0.812]$, with the mean of $0.122$. This indicates substantial improvement. The mean $\lambda$ computed with the previous method was almost $0.049$, and with our new method exceeded $0.17$.
In addition to improving the estimates, application of the new method mitigated some isolated drops in the estimate that occasionally occur without any apparent reason when using the uniform partition.
One of the most considerable differences can be seen for $a=1.94082$ where we were able to raise the bound on \( \lambda\) from about $-0.63$ to almost $0.18$, and the computation time dropped from about 10 seconds to 0.13 seconds, an almost 100-fold increase in speed.
More generally, computations with the new method were completed in considerably shorter time: 
$36$ minutes ($0.6$ hours) single-process CPU time in total, as opposed to 
$337$ minutes ($5.6$ hours) with the previous approach, an improvement in efficiency almost by $90\%$.


\subsection{Dependence of expansion on the size of the critical neighbourhood}
\label{sec:critical}

To compare results for different sizes of the critical neighbourhood $\Delta = (-\delta, \delta)$, we considered $\delta\in\{0.1,0.01,0.001\}$ and applied Algorithm~\ref{alg:main} to compute lower bounds on the expansion exponents for $1025$ equally spaced values of the parameter $a \in \Omega$. The results are shown in Figure~\ref{fig:crit_width}.

\begin{figure}[htbp]
\centering
\includegraphics[width=\textwidth]{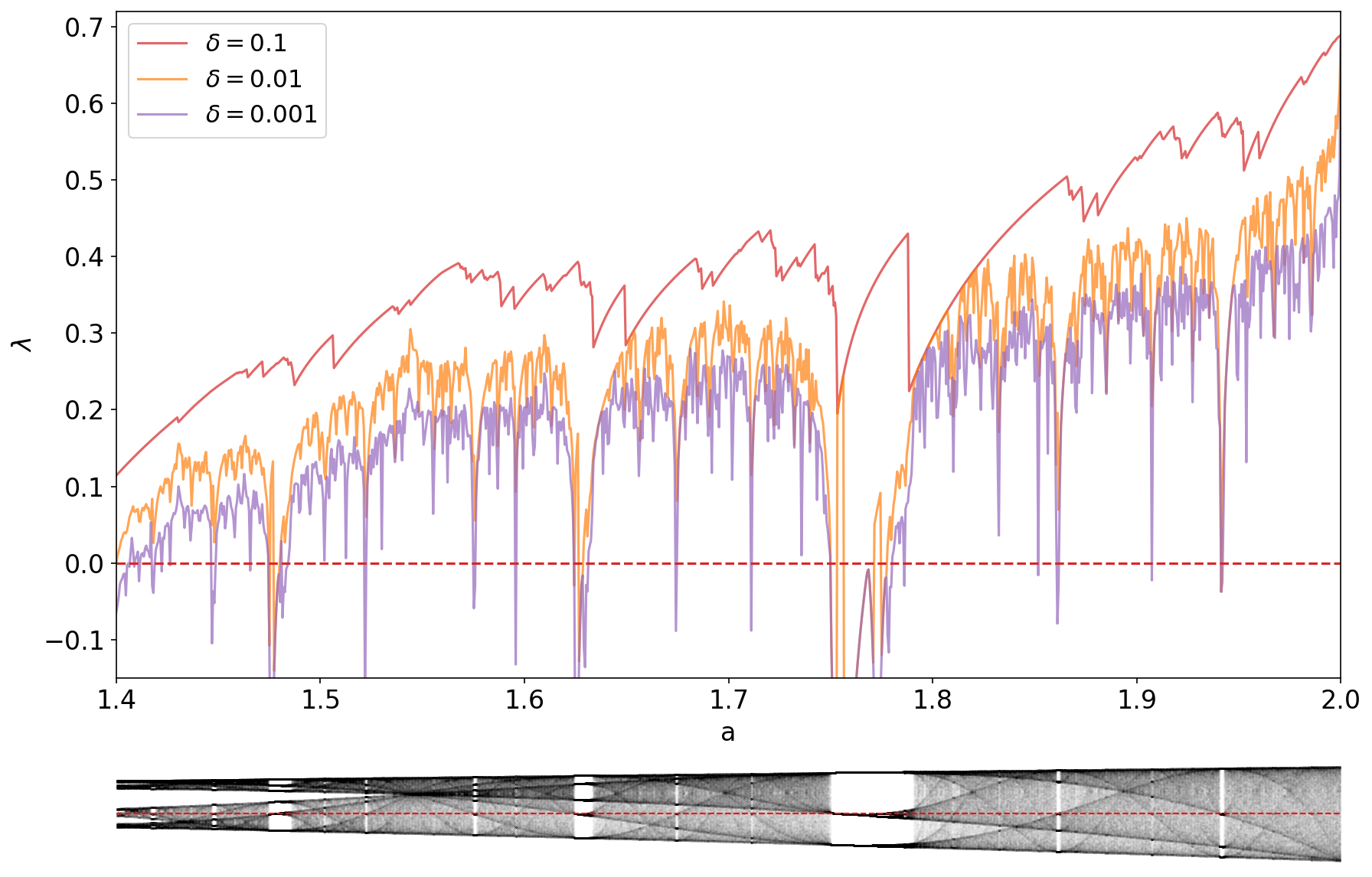}
\caption{The lower bound $\lambda$ on the expansion exponent computed for the quadratic map $f_a$
outside the critical neighbourhood of various radii $\delta > 0$. The bifurcation diagram is shown for reference. (Colour online.)}
\label{fig:crit_width}
\end{figure}

Apart from very few isolated exceptions, the larger the critical neighbourhood the larger the computed expansion \( \lambda \), aligning with the intuitive expectation that small critical neighbourhoods will decrease the overall expansion.


\subsection{Computation for intervals of parameters}
\label{sec:intervals}

Let us now plug in entire intervals of parameters in order to obtain estimates of expansion valid for all the parameters in each interval. It is obvious that such a modification introduces additional underestimates and yields worse results. Let us illustrate this effect.

\begin{figure}[htbp]
    \centering
    \includegraphics[width=\textwidth]{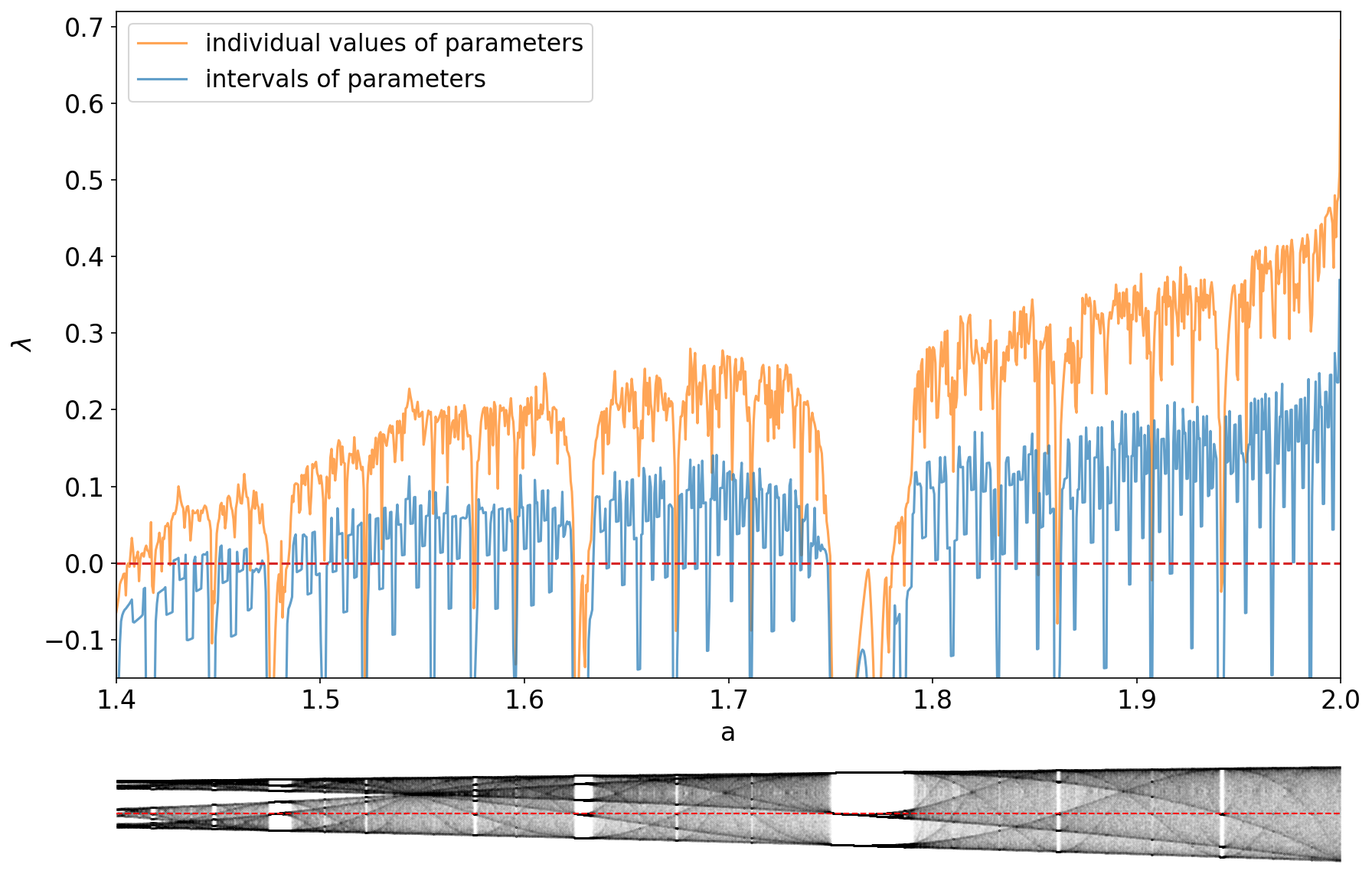}
    \caption{Comparison of the lower bounds on the expansion exponents obtained using rigorous computation for the quadratic map outside the critical neighbourhood of radius $0.001$ for individual values of parameters and intervals of parameters for the quadratic map and the parameter range $[1.4,2]$. (Colour online.)}
    \label{fig:intervals}
\end{figure}

A comparison of the expansion bounds \( \lambda \) computed for $2^{10}$ uniform intervals into which the parameter interval \( \Omega=[1.4,2] \) was split against the individual parameter values taken as the endpoints of these intervals is shown in Figure~\ref{fig:intervals}. It is clear that the estimates obtained for the intervals are substantially worse than those for the individual parameter values.

\begin{figure}[htbp]
    \centering
    \includegraphics[width=0.6\textwidth]{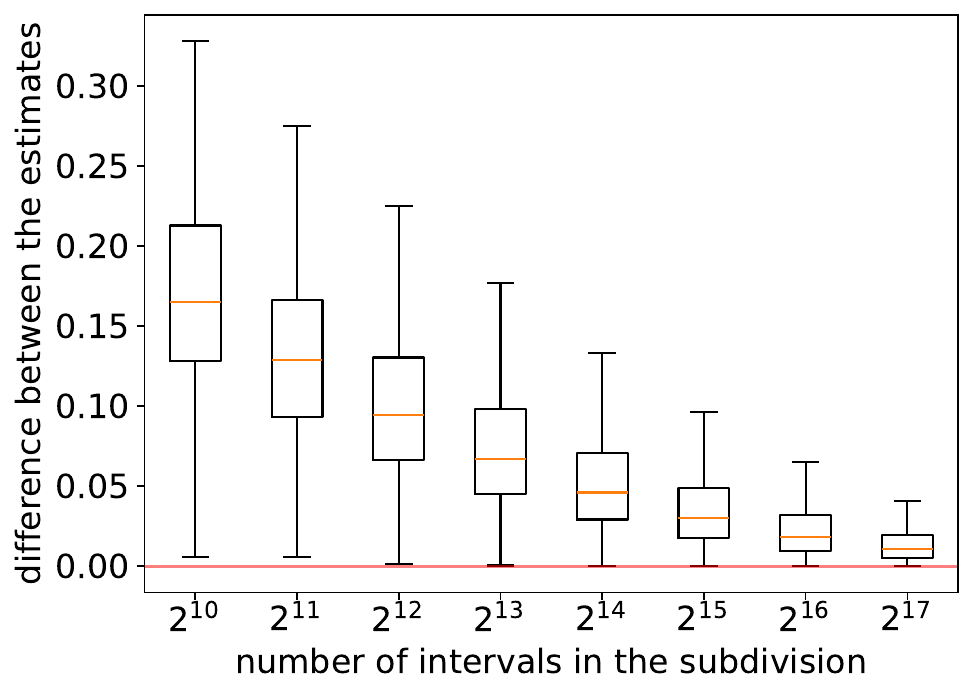}
    \caption{Boxplots of the differences between estimates obtained for single values of the parameter $a$ and for intervals of parameters. Each box spans from the first quartile  to the third quartile of the differences. The whiskers extend to within $1.5$ times the interquartile range from the box.}
    \label{fig:intervals_10_17}
\end{figure}

In Figure~\ref{fig:intervals_10_17} we see that the differences between the expansion estimate obtained for intervals of parameters and those obtained for individual values of the parameters are gradually decreasing with the increase in the number of intervals (and thus the decrease in their widths). We note that the calculations for intervals were longer than for individual parameter values; for example, they took over $16$ hours in the case of $2^{13}$ intervals, compared to less than $4$ hours for the individual parameter values. This difference is mainly due to the higher number of edges in the constructed graphs.


\subsection{Comparison of our method against the suggested derivative-based approach}
\label{sec:derivative}

It was conjectured in~\cite{Day2008} that making the partition finer where the derivative of $f$ is large might yield a better lower bound on the expansion exponent than using a uniform partition.

In order to confirm or refute this statement, we constructed a few types of partition in which the size of intervals was approximately proportional to $1 / \gamma (|f'|)$, where $\gamma \colon \mathbb{R} \to \mathbb{R}$ was a chosen continuous function. We took a few increasing functions (from linear to highly exponential) in order to make the partition scale differently with $|f'|$; namely: $\gamma_1 (x) := x$, $\gamma_2 (x) := x^2$, $\gamma_3 (x) := x^3$, $\gamma_4 (x) := \exp (x)$, and $\gamma_5 (x) := \exp (\exp (x))$. We additionally considered the constant function $\gamma_0 (x) := 1$ that yielded the uniform partition, and the decreasing function $\gamma_{-1} (x) := 1/\exp(x)$ for comparison.

The results of our computations shown in Table~\ref{tab:deriv} indicate that the conjecture stated in~\cite{Day2008} holds true for $a = 2$, but not for most other parameters, for which we obtained results similar to the uniform partition or worse.
The computations were conducted with $\Delta=(-0.001,0.001)$ and $k=5000$ in the case of the uniform or derivative-based partition, and with $k=1000$ for our method. The gain in $\lambda$ was computed for each parameter value for which the uniform partition yielded $\lambda > 0.01$.

\begin{table}[htbp]
\begin{tabular}{r|r|r|r|r}
    $i$ & $\lambda$ for $a=2$ & $\lambda$ for $a=1.9$ & $\lambda$ for $a=1.8$ & mean gain \\
    \hline
    $-1$ & $-0.224268$ & $-0.027449$ & $-0.159342$ & $-0.216370$ \\
    $0$ & $0.235816$ & $0.106392$ & $0.098126$ & $0$ \\
    $1$ & $0.369226$ & $0.106405$ & $0.098153$ & $0.015740$ \\
    $2$ & $0.368526$ & $0.106250$ & $0.097969$ & $-0.047349$ \\
    $3$ & $0.376294$ & $0.106226$ & $0.078889$ & $-0.080425$ \\
    $4$ & $0.379306$ & $0.106384$ & $0.098298$ & $0.024733$ \\
    $5$ & $0.488747$ & $-0.460542$ & $-0.496838$ & $-0.493022$ \\
    \hline
    $\star$ & 0.682346 & 0.290780 & 0.218939 & $0.153606$ \\
\end{tabular}
\bigskip
\caption{\label{tab:deriv}
Lower bound $\lambda$ on the expansion exponent computed using different derivative-based partitions (see the text) in comparison to our method ($\star$), and mean gain in comparison to the uniform partition.}
\end{table}


\subsection{Case study}
\label{sec:case}
We conclude this section with a detailed discussion of how our algorithms work, and how they compare to the uniform partition computations discussed above, for a given parameter value. To make the discussion more interesting, let us study the specific case of $a=1.88516$, because in the computations described in Section~\ref{sec:new_vs_old}, we observed an isolated drop in the expansion exponent for this parameter value, computed using the uniform partition with $k=10\,000$: The expansion exponent for the two neighbouring parameters was slightly above $0.18$, and was negative for the chosen parameter $a$.

The computed negative value of $\lambda$ was most likely a result of some unfortunate misalignment of the partition intervals that caused considerable underestimate in the computation of $\lambda$. In order to confirm this, we conducted computations with some higher values of $k$. Results of these computations (still using the uniform partition) are summarised in Table~\ref{tab:uniformExp}.
\begin{table}[htbp]
\begin{tabular}{r|r|r}
    \multicolumn{1}{c|}{$k$} & \multicolumn{1}{c|}{$\lambda$} & \multicolumn{1}{c}{time} \\
    \hline
    $10\,000$ & $-0.13706844$ & 11 s \\
    $20\,000$ & $0.04802170$ & 45 s \\
    $30\,000$ & $0.06784148$ & 111 s \\
    $40\,000$ & $0.18355391$ & 209 s \\
    $50\,000$ & $0.19460329$ & 263 s \\
    $60\,000$ & $0.20411784$ & 395 s \\
    $70\,000$ & $0.19538890$ & 551 s \\
    $80\,000$ & $0.20539190$ & 689 s \\
    $90\,000$ & $0.20412196$ & 905 s \\
\end{tabular}
\bigskip
\caption{\label{tab:uniformExp}
The lower bound $\lambda$ on the expansion exponent computed using the uniform partition with different values of~$k$.}
\end{table}
For considerably larger partitions, and thus at considerably higher computational cost (both in terms of memory and time usage), we obtained better expansion estimates, although the dependence of $\lambda$ on partition size turned out not to be monotonically increasing.

Table~\ref{tab:dynamicExp} shows the corresponding computations using the dynamically refined partition.
Our method yields considerably better estimates at a dramatically lower cost. Indeed, already at $k=500$, we obtain a better estimate than all the results obtained when using the uniform partition up to $k=90\,000$. Further increase in partition size consistently yields better estimates, and the resulting lower bound $\lambda$ on the expansion exponent stabilizes at about $0.224$. We have reached the accuracy of representable real numbers (IEEE double precision) with $k=2275$, which ended the computation.

\begin{table}[htbp]
\begin{tabular}{r|c|r}
    \multicolumn{1}{c|}{$k$} & \multicolumn{1}{c|}{$\lambda$} & \multicolumn{1}{c}{time} \\
    \hline
    $200$ & $0.128234043654$ & 0.03 s \\
    $300$ & $0.170463327437$ & 0.07 s \\
    $500$ & $0.206783671095$ & 0.25 s \\
    $1\,000$ & $0.223587522152$ & 2.4 s \\
    $2\,000$ & $0.224093873125$ & 24 s \\
    $2\,275$ & $0.224093873133$ & 42 s \\
\end{tabular}
\bigskip
\caption{\label{tab:dynamicExp}
The lower bound $\lambda$ on the expansion exponent computed using our partition with different values of~$k$.}
\end{table}

It turns out that the cycle that was minimising the mean cycle weight was trapped at an unstable periodic orbit of period $7$. Its existence was proved by Algorithm~\ref{alg:periorbit} that also provided $\lambda_{\max} \approx 0.224093873138$, which confirmed that our result was optimal.

\subsection{Iterative subdivisions of subintervals}
\label{sec:iterSubdiv}

It may be interesting to see the actual location of the intervals that form the minimum weight cycles in consecutive iterations in Algorithm~\ref{alg:refine}. We show them in Figure~\ref{fig:r3splitting} for the case discussed in Section~\ref{sec:case}.
The phase space is shown along the vertical axis. The horizontal axis indicates consecutive iterates of the procedure.

At the first iterate, a cycle of two intervals was found to minimise the mean cycle weight bound for $\lambda$. These two intervals are shown as the leftmost vertical segments. The next partition to consider consists of halves of these intervals and the remaining intervals from the original partition. At the second iterate, a cycle of length $3$ is found, consisting of a half of one of the intervals from the previous cycle and two other (original) intervals. These three intervals are then split in two, and the process continues.

The partition being constructed is gradually refined by halving intervals on cycles that minimise the mean cycle weight. The algorithm initially splits large intervals, but the period-$7$ unstable orbit quickly draws its attention, and the splitting ``stabilises'' at this orbit.

\begin{figure}[htbp]
    \centering
    \includegraphics[width=\textwidth]{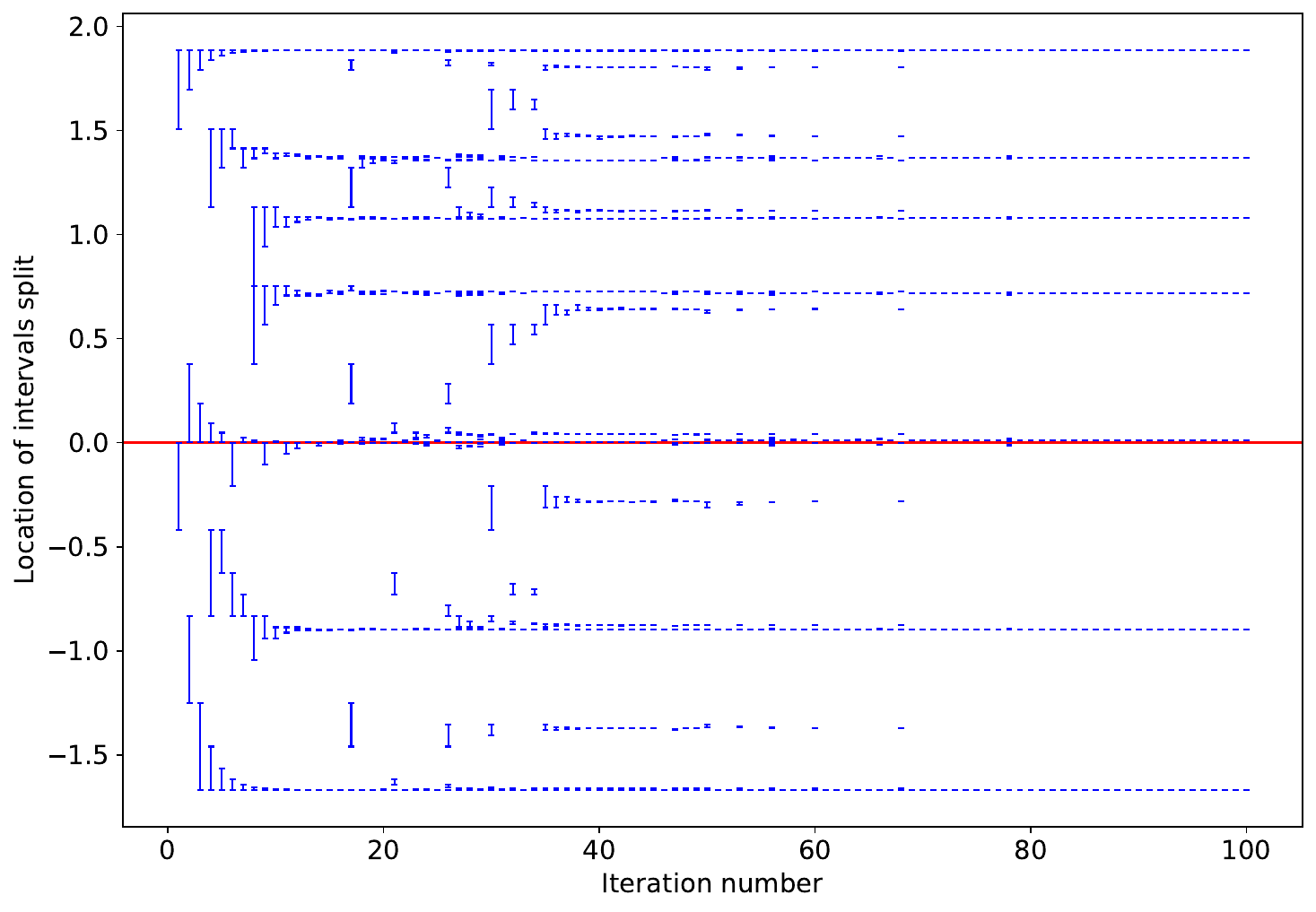}
    \caption{Location of intervals in the phase space $I \setminus \Delta$ that are taken for splitting in the first $100$ iterations of Algorithm~\ref{alg:refine}.}
    \label{fig:r3splitting}
\end{figure}


\section{Optimality of our approach}
\label{sec:periorbit}

In order to show evidence for the optimality of our method, we use a rigorous numerical method designed to prove the existence of a periodic orbit that passes through an explicitly provided interval. If successful, we compute an upper bound \( \lambda_{\max} \) on the expansion exponent along this orbit. The number \( \lambda_{\max} \) is obviously an upper bound on the possible expansion exponent \( \lambda \) that appears in Definition~\ref{def:expC}, provided that the orbit is entirely contained in \( I \setminus \Delta \). In this subsection, we describe the numerical methods that we use for this purpose.


\subsection{Proving the existence of a periodic orbit}

We are going to apply the interval Newton operator introduced in \cite[\S 2.2]{Galias2001}. (Note that this operator was introduced in \cite{Galias2001} for maps $\mathbb{R}^m \to \mathbb{R}^m$, but we only use it in the one-dimensional case.) Given a map $g \colon \mathbb{R} \to \mathbb{R}$, define
\begin{equation}
N([x_1,x_2]) := x_0 - g(x_0) / g'([x_1,x_2]),
\end{equation}
where $x_0 \in [x_1,x_2]$ is arbitrary (one usually chooses the centre of the interval), and $g'([x_1,x_2]) = \{g'(x) : x \in [x_1,x_2]\}$. (We assume that $0 \notin g'([x_1,x_2])$.)

As stated in \cite[Theorem~1]{Galias2001}, if $N([x_1,x_2])$ is contained in the interior of $[x_1,x_2]$ then the equation $g(x)=0$ has a unique solution in $[x_1,x_2]$.
Following the suggestion in \cite[\S 2.4]{Galias2001}, we apply this statement to the map $g(x) := x - f^n(x)$, where $n$ is the period of the orbit whose existence we want to prove, and $[x_1,x_2]$ is an interval which we suspect to contain such an orbit; see Algorithm~\ref{alg:periorbit}.

\begin{algorithm}
    \label{alg:periorbit}
    \caption{Existence of a periodic orbit}
    \SetKwInOut{Input}{input}
    \SetKwInOut{Output}{output}
    \Input{
        $f$: a one-dimensional map $I \to I$; \\
        $\Delta$: a critical neighbourhood for $f$; \\
        $[x_1,x_2]$: an interval contained in $I \setminus \Delta$; \\
        $n > 0$: requested period of the orbit;
    }
    \Output{
        $\emptyset$ if the existence of a periodic orbit was not proved; otherwise: \\
        $\lambda_{\max}$: an upper bound on the expansion exponent along the orbit;
    }

    $x_0 := (x_1 + x_2) / 2$

    use interval arithmetic to compute the following intervals: \\
    $\quad$ $F^n_{x_0}$ containing $f^n (x_0)$, \\
    $\quad$ $G_{x_0}$ containing $g(x_0) := x_0 - f^n(x_0)$, \\
    $\quad$ ${F^n_{x}}'$ containing $(f^n)'(x)$ for all $x \in [x_1,x_2]$, \\
    $\quad$ $G_{x}'$ containing $g'(x) = 1 - (f^n)'(x)$ for all $x \in [x_1,x_2]$;

    \If{$G_{x}'$ contains $0$}{
        \KwRet{$\emptyset$}; (failure: cannot compute the interval Newton operator)
    }

    use interval arithmetic to compute the interval \\
    $\quad$ $N_x$ containing $x_0 - g(x_0)/g'(x)$ for all $x \in [x_1,x_2]$;

    \If{$N_x$ is not contained in the interior of $[x_1,x_2]$}{
        \KwRet{$\emptyset$}; (failure: the inclusion assumption is not satisfied)
    }

    \If{${F^n_{x}}'$ contains $0$}{
        \KwRet{$\emptyset$}; (failure: the expansion exponent cannot be determined)
    }

    use interval arithmetic to compute the intervals \\
    $\quad$ $F^i_x$ containing $f^i([x_1,x_2])$ for all $x \in [x_1,x_2]$, for $i = 1, \ldots, n$

    \If{$F^i_x \cap \Delta \neq \emptyset$ for any $i \in \{1, \ldots, n\}$}{
        \KwRet{$\emptyset$}; (failure: the trajectory might intersect $\Delta$)
    }

    use interval arithmetic to compute \\
    $\quad$ $\lambda_{\max}$ := an upper bound on $\ln(|(f^n)'(x)|)/n$ for all $x \in [x_1,x_2]$;

    \KwRet{$\lambda_{\max}$}
\end{algorithm}

\begin{proposition}
\label{prop:periorbit}
If Algorithm~\ref{alg:periorbit} applied to a one-dimensional map $f \colon I \to I$ of class $C^1$ on $I \setminus \Delta$ does not return $\emptyset$ then there exists exactly one periodic orbit of period $n$ that passes through a point $x$ in the interior of $[x_1,x_2]$. Moreover,
\begin{equation}
\label{eq:expMax}
|(f^n)'(x)| \leq e^{{\lambda_{\max}} n}.
\end{equation}
As a consequence, if $f$ is uniformly expanding outside $\Delta$ (see Definition~\ref{def:expC}) then any constant $\lambda > 0$ for which \eqref{eq:man} holds true must satisfy
\begin{equation}
\label{eq:lambdaMax}
\lambda \leq \lambda_{\max}.
\end{equation}
\end{proposition}

Proof of this proposition follows immediately from the arguments discussed above and the way $\lambda_{\max}$ is constructed in Algorithm~\ref{alg:periorbit}.


\subsection{Evidence for the optimality of our approach}
\label{sec:optimal}

The upper bound $\lambda_{\max}$ on the expansion exponent in \eqref{eq:man} provided by Algorithm~\ref{alg:periorbit} may be used to assess the optimality of the lower bound $\lambda$ computed by Algorithm~\ref{alg:main}.

We applied Algorithm~\ref{alg:main} to $2^{16}$ equally spaced parameter values $a \in [1.4,2]$ with $\Delta=(-0.001,0.001)$ and $K=2\,000$. The lower bound $\lambda$ on the expansion exponent of at least $0.1$ was found in $77.3\%$ of all the cases, reaching the maximum value of $0.687$ and the mean of $0.267$. The existence of periodic orbits was proved for $77.7\%$ of these cases. The periods of the orbits found were between $3$ and $78$, with the median of $17$.

In many cases, especially with large expansion, which we are interested in, the differences between $\lambda$ and $\lambda_{\max}$ were very small.
%
For example, for $a \approx 1.94931641$, the program computed $\lambda \approx 0.35083$ and $\lambda_{\max} \approx 0.35482$ for an unstable periodic orbit of period $19$.
In fact, the differences were below $0.00001$ in almost $31\%$ of the cases, and below $0.01$ in $73\%$ of the cases.


\section{Technical remarks}
\label{sec:tecrem}

In this section, we provide complexity analysis of the algorithms and description of their implementation, as well as information on the environment in which we were running the computations.


\subsection{Complexity analysis of the algorithms}
\label{sec:complexity}

The pessimistic time complexity of Algorithm~\ref{alg:main} is at least cubic in the desired partition size $K$, which follows from the time complexity of Karp's algorithm $O(|V| |E|)$ combined with calling it repeatedly in the selective partition refinement procedure (Algorithm~\ref{alg:refine}).
The space requirement is quadratic, which can be seen in Algorithm~\ref{alg:Karp}, where we need to store the weights of paths of lengths from $1$ to $|V|$ from a chosen source vertex to all the other vertices in the graph.

We show the gradual improvement in $\lambda$ with the running times and sizes of partitions for $a\in\{1.5,1.6,1.7,1.8,1.9,1.95,2\}$ in Figure~\ref{fig:time_lyap}. 
The time needed is relatively low and the results stabilise quickly. Although the estimates obtained are different, the times in which the estimation values stabilise are similar, even though the partition sizes grow steadily.

\begin{figure}[htbp]
    \centering
    \includegraphics[width=\textwidth]{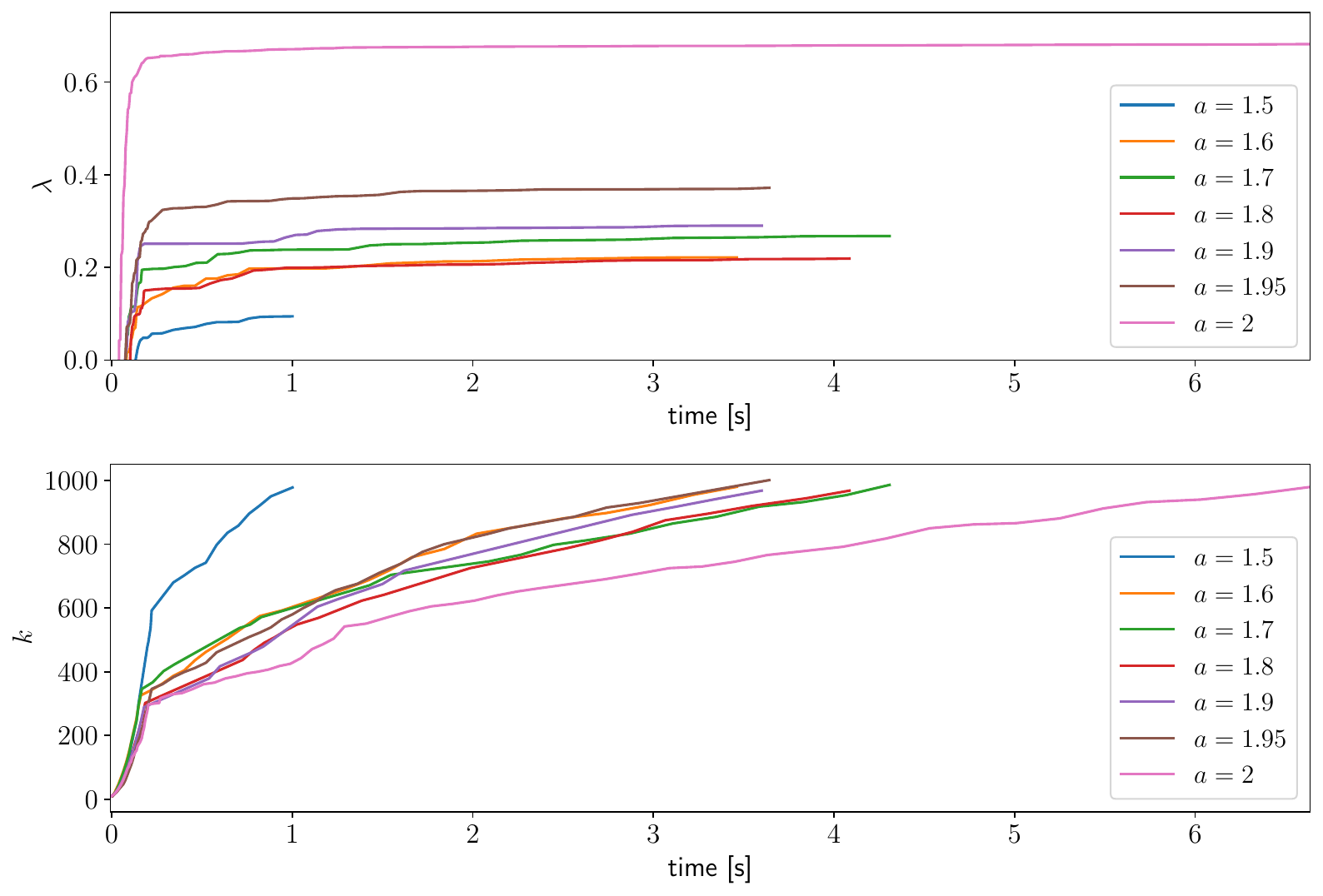}
    \caption{The computed lower bound $\lambda$ on the expansion exponent as a function of time (upper plot), and the size of the constructed partition $k$ as a function of time (lower plot) for rigorous computation for the quadratic map outside the critical neighbourhood of radius $0.001$ for individual values of parameters for the quadratic map for a few selected values of the parameters in the range $[1.5,2]$. (Colour online.)}
    \label{fig:time_lyap}
\end{figure}


\subsection{Software implementation of the method}
\label{sec:software}

The algorithms introduced in Sections \ref{sec:graph} and~\ref{sec:digraph} have been implemented in terms of software written in C++.
The source code is provided in~\cite{www}.
The C++ programming language was chosen for high performance of the resulting program and reasonable flexibility of the code.

The program uses several software libraries: the Boost C++ libraries (\url{https://www.boost.org/}), the CAPD library (\url{http://capd.ii.uj.edu.pl/}) and the CHomP library (\url{https://www.pawelpilarczyk.com/chomp/}). While the first software library is well known, the other two are less standard, and therefore source code of compatible versions is provided in~\cite{www} along with the source code of the program itself; see also \url{https://www.pawelpilarczyk.com/expansion1d/}.
The program uses rigorous interval arithmetic operations on double precision floating point numbers, implemented in the CAPD library together with interval versions of some elementary functions, such as the natural logarithm.
An implementation of Algorithm~\ref{alg:Karp} used in the program, as well as some system tools were taken from parts of the CHomP library programmed by Paweł Pilarczyk.

The program is capable of running computations on multiple machines or in multiple processes on the same machine, using the parallel computation framework introduced in~\cite[Section~3]{Pilarczyk2010} and implemented in the CHomP library. This feature was especially useful in obtaining data shown in figures 
\ref{fig:new_vs_old}, \ref{fig:crit_width}, \ref{fig:intervals} and \ref{fig:intervals_10_17}.

Whenever multiple processes are used in our computations, we report the sum of CPU time used by each process, which corresponds to running the computations as a single process. The computations were completed on a PC with a 12\textsuperscript{th}{} or 13\textsuperscript{th}{} Generation Intel\textsuperscript{\textregistered}{} Core\texttrademark{} i7 processor.
Note that since computers adjust the CPU clock dynamically depending on load and temperature, the computation times provided should be treated as approximate.


\section{The Probability of Stochastic Parameters}\label{sec:motivation}

We conclude this paper with a relatively in-depth discussion of the importance of the uniform expansion condition \eqref{eq:man} in one-dimensional dynamics, including the importance of having an effective algorithm for the numerical verification of this condition, giving explicit values for the constants \( C \) and \( \lambda\). 
A particularly interesting and representative case study, which also highlights the difference between a qualitative and a quantitative approach, is given by the results on the dynamics of the maps in the one-dimensional real \emph{quadratic family} defined in \eqref{eq:quad} above and to which we have already applied our methods, as discussed above.

In Section \ref{sec:qualitative} we discuss some key qualitative results about this family. In Section \ref{sec:towards} we discuss the limitations of these qualitative results and the challenges which exist in obtaining quantitative results. In Section \ref{sec:stochastic} we focus on one particular question, which is the measure of the set of stochastic parameters in the quadratic family, and explain how the rigorous numerical algorithms and computations in this paper constitute a very important step in addressing this questions. 


\subsection{Qualitative results for the quadratic family}
\label{sec:qualitative}
It is fairly easy to verify that for each parameter \( a \in \Omega\), recall \eqref{eq:Omega}, there exists a compact forward-invariant interval \( I_{a} \subset \mathbb R\) in which all the recurrent dynamics occurs. We say that \( a \) is a \emph{regular} parameter if \( f_{a}\) admits an \emph{attracting periodic orbit} to which Lebesgue almost every point in \( I_{a}\) converges, 
and we say that \( a \) is a \emph{stochastic} parameter if \( f_{a}\) admits an \emph{invariant probability measure \( \mu \) which is absolutely continuous with respect to Lebesgue} and which describes the asymptotic statistics of Lebesgue almost every point in \( I_{a}\). We define the two subsets of \( \Omega \) corresponding to regular and stochastic dynamics by 
\[
\Omega^{-}=\{a\in \Omega: a\text{ is regular\} }
\text{ and } 
\Omega^{+}=\{a\in \Omega: \text{is stochastic\}}.
\]
There are infinitely many parameters in \( \Omega \) which are neither regular nor stochastic, and may exhibit an extraordinary richness and variety of dynamical behaviour, but it was proved in \cite{Lyu02} that \emph{Lebesgue almost every \( a \in \Omega\) belongs to either \( \Omega^{-}\) or \( \Omega^{+}\)}. Remarkably, however, the two sets of parameters have very different topological structures: it was proved in the late 1990s that \( \Omega^{-}\) is \emph{open and dense} in \( \Omega \) \cite{GraSwi97, Lyu97a, Lyu97b} implying that \( \Omega^{+}\) is \emph{nowhere dense}, even though it was already known since the 1980s that it has \emph{positive Lebesgue measure} \cite{Jak81, BenCar85}. 


\subsection{Quantitative results for the quadratic family}
\label{sec:towards}

The results mentioned above are deep and nontrivial results which are nevertheless essentially \emph{qualitative}. They do not specify which parameters are regular and which parameters are stochastic, nor describe detailed quantitative features of the dynamics in either case. It is not known, \emph{not even approximately}, what the \emph{actual} measure of the sets \( \Omega^{-}\) and \( \Omega^{+}\) is, nor even if one set is larger than the other. 
Moreover, naive numerical studies do not help at all to answer these questions. There are a handful of small subintervals of \( \Omega\), amounting to a total measure of only about 5\% of the measure of \( \Omega \), where it can be proved relatively easily, both analytically and numerically,
that there exists an attracting periodic orbit of very low period which is also clearly visible even with coarse numerical studies. However, \emph{for every other parameter value}, that is, for about 95\% of parameters in \( \Omega \), \emph{there is no easy way to distinguish between regular and stochastic parameters.} 
This seems paradoxical since from a formal mathematical point of view regular and stochastic dynamics are completely different phenomena. Nevertheless, it is easily seen, for example from the well known \emph{bifurcation diagram} that apart from the few obvious ones mentioned above, all other parameters look just as \emph{chaotic} as each other. 

There are several ways to approach the problem of obtaining quantitative results for the quadratic family and for other similar systems, and in particular of distinguishing between regular and stochastic parameters. 


\subsubsection{Computation of regular parameters}
\label{sec:regular}

The most direct is arguably to determine \emph{regular} parameters by verifying directly the existence of attracting periodic orbits, either analytically or numerically, since this can, in principle, be done in finite time with finite numerical resolution. This is in stark contrast to the situation for stochastic parameters which, apart from a few very exceptional cases, cannot even in principle be verified by a finite number of iterations however high the resolution, see for example \cite{ArbMat04} where it was shown that stochastic parameters are formally \emph{undecidable} (but we will discuss alternative strategies for stochastic parameters below). For regular parameters, rigorous numerical computations in \cite{Gal17, TucWil09} confirm previous non-rigorous calculations in \cite{SimTat91} and provide \emph{explicit} intervals of regular parameter values for the \emph{logistic} family \( f_{\lambda}(x)=\lambda x(1-x) \) (which is smoothly conjugate by an explicit reparametrization to the quadratic family \eqref{eq:quad} above) in the corresponding relevant parameter interval\footnote{Similarly to our comment above, it would perhaps be more precise to define \( \Gamma= [\tilde\lambda, 4]\) where again \( \tilde\lambda \) is the Feigenbaum parameter which, on the logistic family, is known approximately as \( \tilde\lambda\approx 3.5699...\) but this small difference is not so relevant to our discussion.} \( \Gamma := [3.57,4] \). The total measure of these intervals provides a rigorous lower bound of \( |\Gamma^{-}|\geq 0.0455 \) for the corresponding set of regular parameters,\footnote{The formulation of the estimates in \cite{Gal17} is given with reference to the parameter interval \( [3,4]\) and therefore the numbers given there differ from those mentioned here, but match up keeping in mind that all parameters in \([3, 3.5699]\) are already well known to be regular. } which is just over \( 10\% \) of the measure of \( \Gamma\). While the calculations have not been carried out for the quadratic family in the form that we consider here, taking also into account the smooth conjugacy between the two families, it can be expected, and to some extent easily computed, that they would yield very similar results. It seems that even a huge increase in computational power and time would add only a negligible additional measure to the bound already obtained and a heuristic argument is given in \cite{Gal17} to support the conjecture that this bound is very close to the \emph{actual} measure of \( \Gamma^{-}\) but so far there is no formal proof of this fact. Thus, at the moment, \emph{for almost \( 90\%\) of parameters we cannot distinguish between regular and stochastic.}


\subsubsection{Finite resolution dynamics}
\label{sec:finresdyn}

Given the above remarks on the difficulty of differentiating between regular and stochastic behaviour, even with substantial computational power, an interesting point of view is based on the idea that what matters is perhaps not the true nature of the dynamics for a given parameter value but rather what we can \emph{observe}. A parameter \( a\in \Omega^{-}\) may have an attracting periodic orbit of period one hundred (or one million!) and so is strictly speaking a \emph{regular} parameter, in which the asymptotic dynamics of most points is eventually essentially periodic, but to all practical effects and purposes, unless we have an extremely high resolution, the dynamics is essentially indistinguishable from that of a stochastic parameter. A framework was developed in \cite{LuzPil11} which formalizes this point of view and gives a coherent and rigorous mathematical meaning to defining certain dynamical features \emph{at a given resolution}. For example, the dynamics of a map \( f_{a}\) for a regular parameter is not topologically mixing but it may be possible to rigorously prove, using rigorous numerics and rigorous mathematical definitions as described in \cite{LuzPil11}, that it is \emph{topologically mixing at some resolution \( \epsilon\)}. This framework is not restricted to one-dimensional maps, and a classical problem concerns the properties of the map \( H_{a,b}(x,y) = (1+y - ax^{2}, by)\) for \( a=1.4, b=0.3\) studied by H\'enon \cite{Hen76} in the 1970s, and in particular whether it has a periodic or a stochastic attractor. Rigorous computations have proved that for extremely nearby parameter values the map has a periodic attractor \cite{ArbMat04} but the question remains open for \( a=1.4, b=0.3\) and may indeed be \emph{undecidable} if the map is stochastic. However, the paper \cite{LuzPil11}, which formalizes the notion of finite resolution dynamics includes a rigorous computational proof that the H\'enon map for the classical parameter values \( a=1.4, b=0.3\) contains an attractor which is \emph{topologically mixing at all resolutions} \( > 10^{-5}\).


\subsubsection{Probabilistic approach}
\label{sec:propapp}

Yet another approach, which is of main interest for us in this paper and is focused on identifying stochastic parameters, albeit within the limits of what can be achieved keeping in mind their ``undecidability'' mentioned above \cite{ArbMat04}, is a \emph{probabilistic} approach. Given a parameter \( a\in \Omega\), can we at least talk about the \emph{probability} that it is either regular or stochastic? One way to formalize this is to consider a small neighbourhood \( \omega\) of the parameter value \( a\in \Omega\) and let \( \omega^{+}= \omega\cap \Omega^{+}\) be the corresponding subset of stochastic parameters in \(\omega\). If we can bound the measure of \( \omega^{+}\) we can justifiably use this as a bound for the probability that \( a \) is stochastic. For example, if we can show that \( |\omega^{+}|\geq 0.9 |\omega|\) then we can legitimately say that there is a 90\% probability that \( a \) is a stochastic parameter. Given the theoretical impossibility of verifying directly that a given parameter is stochastic, this approach seems to be the best that can be done. Moreover, obtaining such estimates for many small intervals \( \omega\subset \Omega\) would contribute to build up at least a lower bound for the overall measure of the set \( \Omega^{+}\). 


\subsection{The probability of stochastic parameters}
\label{sec:stochastic}

The probabilistic approach outlined in Section \ref{sec:propapp} is of main interest for us in this paper and constitutes the main motivation for the development of the rigorous numerical techniques developed and explained here. We give here some details of how it works and some review of existing results in order to motivate the specific estimates in which we are interested which we define in Section \ref{sec:unifexp}.


\subsubsection{Qualitative parameter-exclusion arguments}
\label{sec:qualEx}

The first proof that \( |\Omega^{+}|>0\) dates back to the early 1980s and is due to Jakobson \cite{Jak81}. This was followed by an alternative proof by Benedicks and Carleson \cite{BenCar85}, and then by a number of generalizations using similar arguments to prove the abundance of stochastic parameters in more general families of one-dimensional maps \cite{Ryc88, ThuTreYou94}, including maps with multiple critical points \cite{Jak04, PacRovVia99, Tsu93, WanYou06} and possibly singularities \cite{LuzTuc99, LuzVia00}. 

The first step in the proofs consists of formulating some condition which implies that \( a \) is stochastic, and the key issue here is that, apart from a handful of exceptional cases, this depends on \emph{all} iterates of the map, and therefore \emph{it cannot be verified in finite time}. We therefore need to formulate a sequence of conditions \( (\star)_{n}\) on the iterates \( f^{n}_{a}\) of the map \( f_{a}\) such that if \( (\star)_{n}\) holds for every \( n \geq 1 \) then \( a\in \Omega^{+}\). For an interval \( \omega\subset \Omega\) we then let \( \omega^{(n)}\) be the subset of parameters which satisfy condition \( (\star)_{i}\) for all \( i \leq n \), and get a nested sequence \( \cdots \subseteq \omega^{(n)}\subseteq \omega ^{(n-1)}\subseteq \cdots \subseteq \omega\) so that \( \cap_{n}\omega^{(n)} \subseteq \omega^{+}\).
It is therefore enough to show that \( |\cap_{n}\omega^{(n)}|>0 \) which is achieved by bounding the measure of the parameters which are \emph{excluded} at each step \( n \), thus showing that some positive measure set survives all exclusions. This is sometimes called a \emph{parameter-exclusion} argument and involves a combination of delicate analytic and probabilistic arguments. 


\subsubsection{Quantitative parameter-exclusion arguments}
\label{sec:quantEx}

All the proofs mentioned above that \( |\Omega^{+}|>0\) are \emph{qualitative}, in the sense that they do not provide any explicit bounds for the actual measure of the set of stochastic parameters. 
In \cite{Jak01, Jak04}, Jakobson posed the question of the possibility of obtaining explicit estimates for the measure of \( \Omega^{+}\) and laid out a theoretical framework for achieving this. 
The challenges are highly nontrivial as it is not just a matter of ``keeping track of the constants.'' The many constants which appear in the parameter-exclusion arguments all fundamentally depend on constants related to certain dynamical properties of special parameters which are used as ``starting points.'' These constants are known to exist by abstract qualitative results but are generally not known explicitly, including those related to the so-called property of ``\emph{uniform expansion outside critical neighbourhoods}'' which we have discussed in detail in Section~\ref{sec:unifexp} and the computation of which is the main purpose of this paper. 

The first explicit lower bound for the measure of \( \Omega^{+}\) was obtained in \cite{LuzTak} where it was proved that \( |\Omega^+|\geq 10^{-5000} \). This is of course an extremely small (and very far from optimal) bound related to the fact that the estimates require working with a very small interval \( \omega\) of parameter values, of size \( 10^{-4990}\), close to the special parameter value \( a=2\). The corresponding map \(f_{2}(x) = 2-x^{2}\) is a so-called \emph{Chebyshev polynomial} and is smoothly conjugate to a piecewise affine ``tent'' map. This can be used to obtain some explicit quantitative estimates (in particular about the constants related to the uniform expansion outside critical neighbourhoods) which make it possible to get an explicit lower bound for the measure of \( \omega^{+}\). 
Improvements on these estimates have been announced independently by Yu-Ru and Shishikura \cite{Shi12} but neither of these works have been published or even circulated as preprints. 


\subsubsection{Role in parameter-exclusion arguments}
\label{sec:unifEx}

The results mentioned in Section~\ref{sec:unifexp} have many consequences and corollaries and are arguably amongst the most fundamental technical results in one-dimensional dynamics, on which many deep arguments are based.
In particular, the parameter-exclusion arguments discussed above depend in an \emph{essential} way on the property that for some special parameter values \emph{the constant \( \lambda \) can be chosen uniformly, independently of the size of the critical neighbourhood~\( \Delta\)}. Moreover, the uniform expansion is an open condition and so, once \( \Delta \) is fixed, condition \eqref{eq:man} continues to hold (for some possibly slightly smaller value of \( \lambda\)) in some neighbourhood \( \omega \) (whose size depends very much on \( \Delta \), and which can be very small if \( \Delta \) is small) of the given parameter. 

Summarising, the key property on which the parameter-exclusion arguments are based is the fact that in a sufficiently small neighbourhood of a ``good'' parameter value \( a^{*}\in \Omega\), condition \eqref{eq:man} holds uniformly, for all parameters in \( \omega \) for some sufficiently small critical neighbourhood \( \Delta \). Philosophically, this can be interpreted as saying that if you have sufficiently strong uniform hyperbolicity, i.e.\ \emph{sufficiently large \( \lambda\)}, relative to a sufficiently large region of the phase space, i.e.\ \emph{sufficiently small \( \Delta\)}, and to a sufficiently large interval of parameters, i.e.\ \emph{sufficiently large \( \omega\)}, then you have non-uniform hyperbolicity in all of the phase space, i.e.\ stochastic behaviour, for a large, i.e.\ positive measure, subset of parameters. 

Choosing \( \omega \) to be a neighbourhood of a ``good'' parameter value, where in particular \( \Delta\) can be chosen arbitrarily small without affecting \( \lambda\), allows all these conditions to be satisfied and the parameter-exclusion argument to successfully yield a positive measure set of stochastic parameters. It does not, however, in general provide any explicit estimates for the measure of the stochastic parameters \emph{unless we know the actual value of} \( \lambda \). This is generally not known, apart from some very exceptional cases, such as the parameter \( a=2\) in which it is known by analytic arguments that \( \lambda = \ln 2 \) independently of the size of \( \Delta \), and this is indeed the reason why it was possible to obtain some explicit estimates in \cite{LuzTak}, as mentioned above. Even more fundamentally, as mentioned at the beginning of Section \ref{sec:towards}, it is not even known, and indeed to some extent \emph{cannot be known}, exactly which parameters are stochastic and therefore even more, which of these are ``good'' parameters which can be used as starting points for the construction.
Nevertheless, a method for finding an explicit lower bound on $\lambda$ for a specific interval $\omega \subset \Omega$ of parameters with a fixed $\Delta$, combined with other methods such as finding parameter intervals $\omega$ for which critical orbits can be iterated for a long time before they hit $\Delta$, addressed in \cite{Golmakani2020}, may become a starting point for obtaining an explicit positive lower bound on the measure of stochastic parameters present in $\omega$, using a method that further develops the ideas introduced in~\cite{LuzTak}.


\section{Conclusion and further work}
\label{sec:conclusions}

The strategies developed in this paper make a significant contribution to the toolbox of rigorous numerical techniques for the study of one-dimensional dynamical systems with discrete time. 
Our immediate motivation, as explained in Section \ref{sec:motivation}, is the application of these techniques to the explicit estimation of the measure of the set of stochastic parameters in the quadratic family, which is currently work in progress. 

The rigorous numerical computation of the expansion exponent of a map outside a critical neighbourhood is however of much more general interest since this property is at the core of many results in dynamics. It is therefore also interesting to ask the extent to which we can apply our techniques, or extend the techniques themselves, to larger classes of maps beyond the one-dimensional quadratic family. We discuss below two different directions in which it would be interesting to extend our results, one is almost immediate while the other is highly non-trivial. 

\subsection{Extensions to other one-dimensional maps}
Our numerical algorithms are of much greater applicability than just the quadratic family as they generalize easily to pretty much any explicitly defined one-dimensional map or family of maps. The exact conditions are stated formally in Section \ref{sec:assumptions} and easily verified in the quadratic family, as discussed in the second paragraph of Section \ref{sec:analysis}. Some classes of maps of interest (e.g., with discontinuities) were listed in \cite[Section~2.1, Figure~1]{Day2008}.

\subsection{Extensions to higher-dimensional maps}
It would be very interesting to extend our techniques to the higher-dimensional setting, replacing the notion of \emph{expansion} with that of \emph{hyperbolicity}, reflecting the fact that in higher dimensions we may have a combination of contraction and expansion.

This would be very useful and relevant in a number of settings, including possible extensions of the work of Benedicks and Carleson on the two-dimensional H\'enon family \cite{BenCar91} which generalizes their earlier work \cite{BenCar85} for the one-dimensional quadratic family which motivates our current paper. The uniform hyperbolicity outside a critical neighbourhood plays again a crucial role in that paper, but can only be verified, and again only \emph{qualitatively}, for parameters for which the H\'enon family is \emph{strongly dissipative}, i.e. ``sufficiently close'' to the one-dimensional quadratic family. In this case the uniform hyperbolicity outside the critical neighbourhood is essentially \emph{inherited} from the uniform expansion of the one-dimensional family, since it is an \emph{open} condition, which itself follows from abstract qualitative Man\'e's Theorem mentioned above. 
The lack of tools to verify the property of uniform hyperbolicity outside a critical neighbourhood is arguably one of the main obstructions in extending the Benedicks-Carleson parameter exclusion argument to more general H\'enon maps which are not strongly dissipative. 

The reason our techniques do not immediately extend to the higher-dimensional, or even to the two-dimensional, setting is in part due to the same reason why the H\'enon family is \emph{much} more difficult to study than the quadratic family. The derivative of iterates of a map along an orbit is given by the \emph{composition of linear maps}, not just multiplication of scalars. The norm of the composition of linear maps is only sub-multiplicative, and therefore any effective control of the hyperbolicity also needs to control the ``rotational'' component of the linear maps involved. The standard method for this is to prove the existence of an \emph{invariant conefield}. This method has been very effective in relatively abstract settings but less so when explicit quantitive estimates are required, although it has been successful in some cases such as for the well known \emph{standard map} \cite{BloLuz09} where the quantitive estimates can be obtained algebraically and without any numerical computations. 

Having said this, there are papers in the literature in which rigorous numerical computations have been used for proving uniform hyperbolicity in two dimensions, see for example \cite{Ara07, AraIsh18, Hru06, MazTab11, Wil10}. None of these methods are, however, ``purely numerical,'' such as the calculations we carry out in this paper, but rather depend on a combination of numerical, analytic, and topological arguments. The authors of \cite{Ara07, AraIsh18, MazTab11} numerically verify quasihyperbolicity, which implies hyperbolicity on the chain recurrent set or on the attractor. In \cite{Hru06}, the author focuses on the verification of hyperbolicity of an invariant set and reduces it to the verification of whether some quadratic forms are positive definite. The author of \cite{Wil10} proves uniform hyperbolicity of an attractor using some indirect arguments. This is different from the setting with which we are concerned, where we want to prove uniform hyperbolicity estimates for finite pieces of orbits whose full orbit may not necessarily be hyperbolic. In fact, some sections of the orbit may be contracting, yet we obtain positive average expansion exponent, as the contraction is compensated by strong expansion along other sections of the orbit. Moreover, out of the mentioned methods only \cite{MazTab11} provides rigorous \emph{explicit} bounds for the \emph{hyperbolicity constants}, but requires strong assumptions, such as semihyperbolic splitting with respect to locally fixed bases in the tangent bundle; see \cite[Theorem, page 1177]{MazTab11}.

It would therefore be very interesting to develop a generalization of our methods to the two-dimensional setting. 

\subsection{Non-uniform subdivision of the parameter interval}
\label{sec:adaptiveParam}

In our method, we propose to partition the phase space $I \setminus \Delta$ in a non-uniform dynamically motivated way (see Section~\ref{sec:partConstr}). However, when considering a parametrised family of maps $f_a$, with $a \in \Omega$, we use uniform partitions in the parameter space $\Omega$ (see Section~\ref{sec:intervals}). We are convinced that there might be some room for improvement if one employs an iterative method of subdividing the parameter space in search for optimal results. A parallelization framework that might be used for such a task was proposed in~\cite{Pilarczyk2010}.

A general idea might be to begin with a coarse subdivision of the parameter space. Then one would compute a lower bound on the expansion exponent for each such interval, constructing a different partition of $I \setminus \Delta$. Then one would split those parameter intervals for which the results are considered not satisfactory using some specific criteria. For example, if the existence of a periodic orbit with low expansion is proved for some parameter (see Section~\ref{sec:periorbit}), then an interval of parameters not containing that specific parameter might be taken in the hope to obtain a higher expansion bound. We leave the details for further research.


\section*{Acknowledgements}
We express our gratitude to Zbigniew Galias for his remark on unstable periodic orbits that motivated us to add Section~\ref{sec:periorbit}.

This research was supported by the National Science Centre, Poland, within the grant OPUS 2021/41/B/ST1/00405.
Some computations were carried out at the Centre of Informatics Tricity Academic Supercomputer \& Network.


\section*{Data availability statement}

Raw data used to create Figures \ref{fig:new_vs_old}, \ref{fig:crit_width}, \ref{fig:intervals} and \ref{fig:intervals_10_17} is publicly available in~\cite{datasets}. The complete data and programs that yield the results of the computations discussed in the paper, as well as scripts for creating all the figures based on the data, are available in~\cite{www}.


\end{document}